\documentclass{amsart}

\usepackage{amssymb}
\usepackage{amsfonts}
\usepackage{amsmath}
\usepackage{amsthm}
\usepackage{graphicx}
\usepackage{mathrsfs}
\usepackage{dsfont}
\usepackage{amscd}
\usepackage{multirow}
\usepackage[all]{xy}
\usepackage[scaled]{helvet} 
\usepackage{courier} 
\normalfont
\usepackage[T1]{fontenc}
\usepackage{calligra}
\usepackage{verbatim}
\usepackage[usenames,dvipsnames]{color}
\usepackage[colorlinks=true,linkcolor=Blue,citecolor=Violet]{hyperref}
\usepackage{enumerate}

\newcommand{\N}{{\mathds{N}}}
\newcommand{\Z}{{\mathds{Z}}}

\newcommand{\R}{{\mathds{R}}}
\newcommand{\C}{{\mathds{C}}}
\newcommand{\T}{{\mathds{T}}}

\newcommand{\A}{{\mathfrak{A}}}

\newcommand{\bigslant}[2]{{\raisebox{.2em}{$#1$}\left/\raisebox{-.2em}{$#2$}\right.}}
\newcommand{\Nbar}{\overline{\N}}
\newcommand{\Lip}{{\mathsf{L}}}
\newcommand{\Hilbert}{{\mathscr{H}}}

\newcommand{\Kantorovich}[1]{{\mathsf{mk}_{#1}}}

\newcommand{\StateSpace}{{\mathscr{S}}}

\newcommand{\MongeKant}{{Mon\-ge-Kan\-to\-ro\-vich metric}}

\newcommand{\Lqcms}{{\JLL} quantum compact metric space}

\newcommand{\gQqcms}{quasi-Leibniz quantum compact metric space}

\newcommand{\gQVB}{metrized quantum vector bundle}

\newcommand{\unit}{1}

\newcommand{\sa}[1]{{\mathfrak{sa}\left({#1}\right)}}

\newcommand{\inner}[3]{{\left<{#1},{#2}\right>_{#3}}}
\newcommand{\JLL}{Lei\-bniz}

\newcommand{\dom}[1]{{\operatorname*{dom}\left({#1}\right)}}

\newcommand{\diam}[2]{{\mathrm{diam}\left({#1},{#2}\right)}}

\newcommand{\norm}[2]{{\left\|{#1}\right\|_{#2}}}

\newcommand{\CDN}{{\mathsf{D}}}
\newcommand{\cocycle}[1]{{\mathrm{e}_{#1}}}
\newcommand{\worknote}[1]{}
\newcommand{\opnorm}[3]{{\left|\mkern-1.5mu\left|\mkern-1.5mu\left| {#1} \right|\mkern-1.5mu\right|\mkern-1.5mu\right|_{#3}^{#2}}}

\newcommand{\alg}[1]{{\mathfrak{#1}}}

\newcommand{\modlip}[2]{{\module{D}_{#1}\left({#2}\right)}}
\newcommand{\module}[1]{{\mathscr{#1}}}
\newcommand{\qt}[1]{{\mathcal{A}_{#1}}}
\newcommand{\HeisenbergGroup}{{\mathds{H}_3}}
\newcommand{\HeisenbergMod}[2]{{\module{H}_{#1}^{#2}}}

\theoremstyle{plain}
\newtheorem{theorem}{Theorem}[section]

\newtheorem{corollary}[theorem]{Corollary}

\newtheorem{lemma}[theorem]{Lemma}
\newtheorem{proposition}[theorem]{Proposition}

\newtheorem{theorem-definition}[theorem]{Theorem-Definition}

\theoremstyle{definition}
\newtheorem{definition}[theorem]{Definition}

\newtheorem{convention}[theorem]{Convention}
\newtheorem{hypothesis}[theorem]{Hypothesis}

\theoremstyle{remark}

\newtheorem{remark}[theorem]{Remark}

\newtheorem{notation}[theorem]{Notation}

\renewcommand{\geq}{\geqslant}
\renewcommand{\leq}{\leqslant}

\newcommand{\conv}[1]{{*_{#1}}}

\numberwithin{equation}{section}

\allowdisplaybreaks[4]

\begin{document}

\title[D-norms on Heisenberg Modules]{Heisenberg Modules over Quantum $2$-tori are metrized quantum vector bundles}
\author{Fr\'{e}d\'{e}ric Latr\'{e}moli\`{e}re}
\thanks{This work is part of the project supported by the grant H2020-MSCA-RISE-2015-691246-QUANTUM DYNAMICS}
\email{frederic@math.du.edu}
\urladdr{http://www.math.du.edu/\symbol{126}frederic}
\address{Department of Mathematics \\ University of Denver \\ Denver CO 80208}

\date{\today}
\subjclass[2000]{Primary:  46L89, 46L30, 58B34.}
\keywords{Noncommutative metric geometry, Gromov-Hausdorff convergence, Monge-Kantorovich distance, Quantum Metric Spaces, Lip-norms, D-norms, Hilbert modules, noncommutative connections, noncommutative Riemannian geometry, unstable $K$-theory.}

\begin{abstract}
The modular Gromov-Hausdorff propinquity is a distance on classes of modules endowed with quantum metric information, in the form of a metric form of a connection and a left Hilbert module structure. This paper proves that the family of Heisenberg modules over quantum two tori, when endowed with their canonical connections, form a family of {\gQVB s}, as a first step in proving that Heisenberg modules form a continuous family for the modular Gromov-Hausdorff propinquity.
\end{abstract}
\maketitle



\section{Introduction}

The primary purpose of our research is to bring forth an analytic framework, constructed around Gromov-Hausdorff-like hypertopologies on quantum metric spaces, to bear on problems from mathematical physics and noncommutative geometry \cite{Latremoliere13,Latremoliere13b,Latremoliere14,Latremoliere15,Latremoliere15b,Latremoliere15c,Latremoliere15d,Latremoliere16b}. We constructed an hypertopology on classes of Hilbert modules over quantum metric spaces in \cite{Latremoliere16c} as a far-reaching generalization of the Gromov-Hausdorff propinquity. We constructed a distance, up to full quantum isometry, called the modular Gromov-Hausdorff propinquity, on a class of objects which generalize Hermitian vector bundles over Riemannian manifolds. These {\gQVB s} are natural objects for noncommutative geometry and mathematical physics, as they carry a metric structure and a form of generalized connection, and we are now able to discuss such questions as continuity and approximations, not only of quantum compact metric spaces, but also of their associated modules. As modules are fundamental objects in C*-algebra theory and their geometry, this new development allows us to further our goal of a geometric theory of the class of C*-algebras.

This paper brings into our noncommutative metric geometry framework some very important examples of modules, namely Heisenberg modules over quantum $2$-tori. These modules come naturally equipped with a connection induced by the action of the Heisenberg Lie group. This noncommutative construct played the central role in the beginning of Connes' noncommutative geometry \cite{Connes80}, where the Heisenberg modules over quantum $2$-tori and their connections were first built. Rieffel \cite{Rieffel93} then proved that these Heisenberg modules, the finite rank free modules, and their direct sums, describe all the finitely generated projective modules over quantum tori. Connes and Rieffel \cite{ConnesRieffel87} proved that the natural connections on Heisenberg modules solve the noncommutative Yang-Mills problem. We will now prove that Heisenberg modules are fundamental examples of {\gQVB s}. Doing so then allows us to discuss in \cite{Latremoliere18a} the continuity, for the modular propinquity, of family of Heisenberg modules as the quantum $2$-tori vary continuously for the propinquity. This will be our first, significant application of the modular propinquity. Informally, the continuity result in \cite{Latremoliere18a} can be understood as a form of continuity of K-theory. Thus, this paper and \cite{Latremoliere18a} are two parts of the study of the metric geometry of Heisenberg modules.

As a matter of convention throughout this paper, we will use the following notations.

\begin{notation}
By default, the norm of a normed vector space $E$ is denoted by $\norm{\cdot}{E}$. When $\A$ is a C*-algebra, the space of self-adjoint elements of $\A$ is denoted by $\sa{\A}$. The state space of $\A$ is denoted by $\StateSpace(\A)$. In this work, all C*-algebras $\A$ will always be unital with unit $\unit_\A$.
\end{notation}

\begin{convention}
If $P$ is some seminorm on a vector subspace $D$ of a vector space $E$, then for all $x\in E\setminus D$ we set $P(x) = \infty$. With this in mind, the domain $D$ of $P$ is the set $\{ x \in E : P(x) < \infty \}$, with the usual convention that $0\infty = 0$ while all other operations involving $\infty$ give $\infty$.
\end{convention}

Noncommutative metric geometry \cite{Connes89,Rieffel98a,Rieffel00} studies noncommutative generalizations of Lipschitz algebras, defined as follows.

\begin{definition}
An ordered pair $(\A,\Lip)$ is a \emph{\Lqcms} when $\A$ is a unital C*-algebra, and $\Lip$ is a seminorm defined on a dense Jordan-Lie subalgebra $\dom{\Lip}$ of the space of self-adjoint elements $\sa{\A}$ of $\A$ such that:
\begin{enumerate}
\item $\{a\in\dom{\Lip} : \Lip(a) = 0 \} = \R\unit_\A$,
\item the {\MongeKant} $\Kantorovich{\Lip}$ defined on the state space $\StateSpace(\A)$ of $\A$ by setting, for any two $\varphi,\psi \in \StateSpace(\A)$:
\begin{equation*}
\Kantorovich{\Lip}(\varphi,\psi) = \sup\left\{ |\varphi(a) - \psi(a)| : a\in\dom{\Lip}, \Lip(a) \leq 1 \right\}
\end{equation*}
metrizes the weak* topology restricted to $\StateSpace(\A)$,
\item $\Lip$ is lower semi-continuous,
\item $\max\left\{ \Lip\left(\frac{a b + b a}{2}\right),\Lip\left(\frac{a b - b a}{2i}\right)\right\} \leq \|a\|_\A \Lip(b) + \|b\|_\A \Lip(a)$.
\end{enumerate}
\end{definition}

{\Lqcms s}, and more generally {\gQqcms s} (a generalization we will not need in this paper), form a category with the appropriate notion of Lipschitz morphisms \cite{Latremoliere16}, containing such important examples as quantum tori \cite{Rieffel98a}, Connes-Landi spheres \cite{li03}, group C*-algebras for Hyperbolic groups and nilpotent groups \cite{Rieffel02,Ozawa05}, AF algebras \cite{Latremoliere15c}, Podl{\`e}s spheres \cite{Kaad18}, certain C*-crossed-products \cite{Hawkins13}, among others. Any compact metric space $(X,d)$ give rise to the {\Lqcms} $(C(X),\mathrm{Lip})$ where $C(X)$ is the C*-algebra of $\C$-valued continuous functions over $X$, and $\mathrm{Lip}$ is the Lipschitz seminorm induced by $d$.

Rieffel characterized the main property of {\Lqcms s} as follows:
\begin{theorem}[{\cite[Theorem 1.9]{Rieffel98a}}]\label{az-thm}
Let $(\A,\Lip)$ be a pair with a unital C*-algebra $\A$ and a seminorm $\Lip$ defined on a dense subspace $\dom{\Lip}$ of $\sa{\A}$. The following assertions are equivalent:
\begin{enumerate}
\item the {\MongeKant} $\Kantorovich{\Lip}$ defined for any two $\varphi,\psi\in\StateSpace(\A)$ by $\Kantorovich{\Lip}(\varphi,\psi) = \sup\left\{|\varphi(a) - \psi(a)|:\Lip(a)\leq 1\right\}$, metrizes the weak* topology on $\StateSpace(\A)$,
\item the diameter $\diam{\StateSpace(\A)}{\Kantorovich{\Lip}}$ is finite and:
\begin{equation*}
\left\{a\in\sa{\A} : \Lip(a)\leq 1\text{ and }\|a\|_\A \leq 1 \right\}
\end{equation*}
is norm precompact.
\end{enumerate}
\end{theorem}

In \cite{Latremoliere16c}, we extend this idea to noncommutative analogues vector bundles. Our classical prototype of a {\gQVB} is given by the module $\Gamma V$ of continuous sections of a vector bundle $V$ over a compact Riemannian manifold $M$ with metric $g$, endowed with a hermitian metric $h$ and some associated metric connection $\nabla$. For any two $\omega,\eta \in \Gamma V$, we then set $\inner{\omega}{\eta}{V} : x \in M \mapsto \int_X h_x(\omega_x,\eta_x) \, d\mathrm{Vol}(x)$ where $\mathrm{Vol}$ is the volume form over $M$ for $g$, which turns $\Gamma V$ into a $C(M)$-left Hilbert module. We also define, for all $\omega \in \module{M}$, the norm $\CDN(\omega)$ as the operator norm for the operator $\nabla\omega : X \in \Gamma(TM) \mapsto \nabla_X \omega \in \Gamma V$ --- noting that the space of vector fields $\Gamma TM$ of $M$ has a norm induced by the metric $g$. Our general definition for a {\gQVB} abstracts this picture. For the present paper, we shall only deal with so-called Leibniz {\gQVB s}, even though our framework in \cite{Latremoliere16c} is more general. This is the main definition for this paper.

\begin{definition}[{\cite[Definition 3.8]{Latremoliere16c}}]\label{qvb-def}
A 5-tuple $(\module{M},\inner{\cdot}{\cdot}{\module{M}},\CDN,\A,\Lip)$ is a \emph{\gQVB} when:
\begin{enumerate}
\item a {\Lqcms} $(\A,\Lip)$ called the base space,
\item a $\A$-left Hilbert module $(\module{M},\inner{\cdot}{\cdot}{\module{M}})$,
\item a norm $\CDN$ defined on a dense subspace of $\module{M}$ such that $\CDN(\omega) \geq \sqrt{\inner{\omega}{\omega}{\module{M}}}$ for all $\omega\in \module{M}$, and such that the set:
\begin{equation*}
\left\{ \omega \in \module{M} : \CDN(\omega) \leq 1 \right\}
\end{equation*}
is compact in $\module{M}$,
\item for all $a \in \sa{\A}$ and for all $\omega \in \module{M}$, we have:
\begin{equation*}
\CDN_{\module{M}}\left(a\omega\right) \leq (\|a\|_\A+\Lip_\A(a)) \CDN_{\module{M}}(\omega)\text{,}
\end{equation*}
which we call the \emph{inner Leibniz inequality} for $\CDN_{\module{M}}$,
\item for all $\omega,\eta \in \module{M}$, we have:
\begin{equation*}
\max\left\{\Lip_\A\left(\Re\inner{\omega}{\eta}{\module{M}}\right),\Lip_\A\left(\Im\inner{\omega}{\eta}{\module{M}}\right)\right\} \leq 2 \CDN_{\module{M}}(\omega) \CDN_{\module{M}}(\eta) \text{,}
\end{equation*}
which we call the \emph{modular Leibniz inequality} for $\CDN_{\module{M}}$.
\end{enumerate}
\end{definition}

We refer to \cite{Latremoliere16c} for a discussion of these objects, where in particular \cite[Example 3.10]{Latremoliere16c} shows that the prototype of a hermitian vector bundle over a compact Riemannian manifold, as sketched above, is indeed an example of a {\gQVB}.  We note that Definition (\ref{qvb-def}) includes a compactness condition which mirrors the compactness condition in Theorem (\ref{az-thm}).

Heisenberg modules, equipped with the analogue of a connection as in \cite{Connes80}, over quantum $2$-tori, have a similar signature to a {\gQVB}. The key difficulty is to prove that the connection can be used to define a D-norm, as in Definition (\ref{qvb-def}), whose unit ball is actually compact in the Hilbert modules norm of Heisenberg modules. The main result of  this paper is to prove that indeed, this is the case. 

We begin our work with a presentation of Heisenberg modules, which allow us to fix our notations for the rest of the paper and \cite{Latremoliere18a}. We then prove a series of lemmas about convergence in the Hilbert modules norm for the Heisenberg modules --- as these norms are complicated, these lemmas will prove very helpful both in this paper and in \cite{Latremoliere18a}. We prove in the process of this second section that Heisenberg modules form a continuous field of Banach spaces --- a result which will prove helpful in \cite{Latremoliere18a} and is of independent interest. This result uses the same tools as the proof that the action of the Heisenberg group on Heisenberg modules is strongly continuous, which is part of the next section of this paper, where properties of the Heisenberg group actions which we will need in our work are established. Now, with all these basic tools in hand, we show how to use Lie group actions to define D-norm candidates, which have all the desired properties of D-norms except maybe for the key compactness property of their unit ball. This compactness property for the Heisenberg modules D-norms is the subject of the last section of this paper, which conclude our main result.

Importantly, our methods in this paper are designed not only in support of the main theorem here, but also as key tools for the study of the continuity of the Heisenberg modules in \cite{Latremoliere18a}. For the problem of continuity, we will need not just to be able to pick finite subsets of the compact unit ball of some D-norm which are $\varepsilon$-dense for some $\varepsilon > 0$, but also to pick such a finite set which is uniformly $\varepsilon$-dense across several Heisenberg modules as the D-norms vary. To do so, we will use the approximation operators introduced in the last section of this paper.

\section{Background on Quantum $2$-tori and Heisenberg modules}

Quantum $2$-tori are the twisted convolution C*-algebras of $\Z^2$. The projective finitely generated modules over quantum tori have been extensively studied, and next to the free modules, the most important class of projective, finitely generated modules over a quantum torus are the Heisenberg modules. This subsection introduces these modules, as well as the notations we will use throughout this section regarding quantum tori.

Twisted group C*-algebras are defined by twisting the convolution product over a locally compact group by a representative of a continuous $2$-cocycle of the group.

\begin{notation}
For any $\theta \in \R$, we define the skew bicharacter of $\R^2$:
\begin{equation}\label{sigma-eq}
\cocycle{\theta} : ((x_1,y_1),(x_2,y_2)) \in \R^2\times \R^2 \longmapsto \exp\left(i\pi\theta(x_2 y_1 - x_1 y_2)\right)\text{.}
\end{equation}

By \cite{Kleppner65}, any $2$-cocycle of $\Z^2$ is cohomologous to the restriction of a skew bicharacter $\cocycle{\theta}$ to $\Z^2\times\Z^2$ for some $\theta\in\R$. We shall use the same notation for $\cocycle{\theta}$ and its restriction to $\Z^2$.

Moreover, for any $\theta,\vartheta\in\R$, the skew bicharacters $\cocycle{\theta}$ and $\cocycle{\vartheta}$ of $\Z^2$ are cohomologous if and only if $\theta \equiv \vartheta \mod 1$. We note that, as skew bicharacters of $\R^2$, they are cohomologous if and only if $\theta = \vartheta$.
\end{notation}

We define the twisted convolution products on $\ell^1(\Z^2)$, where we use the following notation.

\begin{notation}
For any (nonempty) set $E$ and any $p \in [1,\infty)$, the set $\ell^p(E)$ is the set of all absolutely $p$-summable complex valued functions over $E$, endowed with the norm:
\begin{equation*}
\|\xi\|_{\ell^p(E)} = \left(\sum_{x \in E} |\xi(x)|^p\right)^{\frac{1}{p}}
\end{equation*}
for all $\xi \in \ell^p(E)$.

We write $\delta_n$ the function which is $1$ at $n$ and $0$ otherwise; this function is an element of $\ell^p(E)$ for all $p$.

Moreover, if $p = 2$ then $(\ell^2(E),\|\cdot\|_{\ell^2(E)})$ is a Hilbert space, where the inner product $\inner{\xi}{\eta}{\ell^2(E)} = \sum_{x\in E} \xi(x)\overline{\eta(x)}$ for all $\xi,\eta\in\ell^2(E)$.
\end{notation}

We now define:

\begin{definition}
Let $\theta\in\R$ and $\cocycle{\theta}$ be defined by Expression (\ref{sigma-eq}). The \emph{twisted convolution product $\conv{\theta}$} is defined for all $f,g \in \ell^1(\Z^2)$ and for all $n\in\Z^2$:
\begin{equation*}
f\conv{\theta} g (n) = \sum_{m\in\Z^2} f(m) g(n - m)\cocycle{\theta}(m,n)\text{.}
\end{equation*}
The \emph{adjoint} of any $f \in \ell^1(\Z^2)$ is defined for all $n\in\Z^2$ by:
\begin{equation*}
f^\ast(n) = \overline{f(-n)}\text{.}
\end{equation*}
\end{definition}

One checks easily that $\left(\ell^1(\Z^2),\conv{\theta},\cdot^{\ast}\right)$ is a *-algebra. In particular, the adjoint operation is an isometry of $\left(\ell^1(\Z^2),\|\cdot\|_{\ell^1(\Z^2)}\right)$. We now wish to construct its enveloping C*-algebra. To do so, we shall choose a natural faithful *-representation of $\left(\ell^1(\Z^2),\conv{\theta},\cdot^{\ast}\right)$ on $\ell^2(\Z^2)$. This representation was a key ingredient in the construction of bridges between quantum tori in our work in \cite{Latremoliere13c} on convergence of quantum tori for the quantum propinquity and will play a role in the convergence of Heisenberg modules.

\begin{notation}
If $T : E\rightarrow F$ is a continuous linear map between two normed spaces, we write its norm as $\opnorm{T}{E}{F}$. When $E=F$, we simply write $\opnorm{T}{}{F}$.
\end{notation}

\begin{theorem}[{\cite{Zeller-Meier68}}]\label{qt-construction-thm}
Let $\theta \in \R$. We define, for any $n \in \Z^2$ and $\xi \in \ell^2(\Z^2)$, the function:
\begin{equation*}
U_\theta^n \xi : m \in \Z^2 \mapsto \cocycle{\theta}(m,n)\xi(m+n) \text{.}
\end{equation*}
The map $n\in\Z^2 \mapsto U_\theta^n$ is a unitary $\cocycle{\theta}$-projective representation of $\Z^2$, i.e. $U_\theta^n U_\theta^m = \cocycle{\theta}(n,m) U_\theta^{n+m}$ for all $n,m \in \Z^2$.

If, for all $f \in \left(\ell^1(\Z^2),\conv{\theta},\cdot^{\ast}\right)$, we define:
\begin{equation*}
\pi_\theta(f) = \sum_{n \in \Z^2} f(n) U_\theta^n
\end{equation*}
which is a bounded operator on $\ell^2(\Z^2)$ with:
\begin{equation*}
\opnorm{\pi_\theta(f)}{}{\ell^2(\Z^2)} \leq \|f\|_{\ell^1(\Z^2)}\text{,}
\end{equation*}
then $\pi_\theta$ is a faithful *-representation of $(\ell^1(\Z^2),\conv{\theta},\ast)$. 
\end{theorem}

Thus, we may define a C*-norm on $\ell^1(\Z^2)$ by setting:
\begin{equation*}
\|f\|_{\qt{\theta}} = \opnorm{\pi_\theta(f)}{}{\ell^2(\Z^2)}
\end{equation*}
for all $f\in \ell^1(\Z^2)$. We thus can define quantum $2$-tori.

\begin{definition}
The \emph{quantum $2$-torus $\qt{\theta}$} is the completion of $(\ell^1(\Z^2),\conv{\theta},\cdot^\ast)$ for the norm $\opnorm{\pi_\theta(\cdot)}{}{\ell^2(\Z^2)}$.
\end{definition}

As per our general convention, the norm on $\qt{\theta}$ is denoted by $\|\cdot\|_{\qt{\theta}}$ for all $\theta \in \R$.

\begin{remark}
Let $\theta \in \R$. By construction, $\ell^1(\Z^2)$ is identified with a dense *-subalgebra of $\qt{\theta}$, and we shall employ this identification all throughout this paper. With this identification, we also note that for all $f \in \ell^1(\Z^2)$ we have $\|f\|_{\qt{\theta}} \leq \|f\|_{\ell^1(\Z^2)}$, a fact which we will use repeatedly in the next section.
\end{remark}

We take one derogation to the convention of using the same symbol for an element of $\ell^1(\Z^2)$ and its counter part in a given quantum torus, because the following notation is at once common and convenient.
\begin{notation}
Let $\theta\in \R$. The element $\delta_{1,0}$ is denoted by $u_\theta$ and the element $\delta_{0,1}$ is denoted by $v_\theta$ when regarded as elements of $\qt{\theta}$.
\end{notation}

The geometry, and in particular the metric geometry \cite{Rieffel98a}, of the quantum tori is obtained by transport of structure using the dual action of the torus given as follows:

\begin{theorem-definition}\cite{Zeller-Meier68}\label{dual-action-thm}
For all $z = (z_1,z_2) \in \T^2$ there exists a unique *-automorphism $\beta_\theta^z$ of $\qt{\theta}$ such that, for any $f\in \ell^1(\Z^2)$ and $(n,m) \in \Z^2$, we have:
\begin{equation*}
\beta_\theta^z f (n,m) = z_1^n z_2^m f(n,m) \text{.}
\end{equation*}
The map $z\in\T^2 \mapsto \beta_\theta^z$ is a strongly continuous action of $\T^2$ on $\qt{\theta}$ called the \emph{dual action}. Moreover, $\beta$ is ergodic, in the sense that:
\begin{equation*}
\left\{ a \in \qt{\theta} : \forall z \in \T^2 \quad \beta^z(a) = a \right\} = \C\unit_{\qt{\theta}}\text{.}
\end{equation*}
\end{theorem-definition}

We now turn to the class of modules to which we shall apply our new modular propinquity. We construct these modules following \cite{Connes80} using the universal property of quantum $2$-tori, which we now recall.

\begin{proposition}[\cite{Zeller-Meier68}]\label{qt-universal-prop}
Let $\theta\in\R$. If $U$, $V$ are two unitary operators on some Hilbert space $\Hilbert$ such that $U V = \exp(2i\pi\theta) V U$ for some $\theta\in [0,1)$, then there exists a *-morphism $\varpi : \qt{\theta} \rightarrow \alg{B}(\Hilbert)$ such that $\varpi(u_\theta) = U$ and $\varpi_\theta(v_\theta) = V$. The range of $\varpi$ is $C^\ast(U,V)$.
\end{proposition}

Another way to state Proposition (\ref{qt-universal-prop}) is that, for any $\theta\in\R$, if $\varsigma$ is some projective representation of $\Z^2$ on some Hilbert space $\Hilbert$ for some multiplier of $\Z^2$ cohomologous to $\cocycle{\theta}$, then $\Hilbert$ is a module over $\qt{\theta}$. Indeed, Proposition (\ref{qt-universal-prop}) gives us a *-morphism $\varpi$ from $\qt{\theta}$ to the C*-algebra $\alg{B}(\Hilbert)$ of all bounded linear operators on $\Hilbert$, with $\varpi(u_\theta) = \varsigma^{1,0}$ and $\varpi(v_\theta) = \varsigma^{0,1}$. Thus $\Hilbert$ is a $\qt{\theta}$ module.

With this observation in mind, we now turn to the construction of some particular projective representations of $\Z^2$. The idea, found in \cite{Connes80} and explicit in \cite{Rieffel88}, is to take the tensor product of a projective representation of $\R^2$, restricted to $\Z^2$, and a finite dimensional projective representation of $\Z_q$ for some $q\in\N\setminus\{0\}$. By adjusting the choice of the multipliers associated with each projective representation, we get the desired module structure.

Projective representations of $\R^2$ are naturally related to the representations of the Heisenberg group, and we will make important use of this fact in our work. We thus begin with setting our notations for the Heisenberg group.

\begin{convention}
The vector space $\C^d$ is endowed by default with its standard inner product $\inner{(z_1,\ldots,z_d)}{(y_1,\ldots,y_d)}{\C^d} = \sum_{j=1}^d z_j\overline{y_j}$, whose associated norm is denoted by $\|\cdot\|_{\C^d}$.
\end{convention}

\begin{notation}\label{Schrodinger-notation}
The Heisenberg group is the Lie group given by:
\begin{equation*}
\HeisenbergGroup = \left\{ \begin{pmatrix} 1 & x & u \\ 0 & 1 & y \\ 0 & 0 & 1 \end{pmatrix} : x,y,u \in \R \right\}\text{.}
\end{equation*}
We shall identify $\HeisenbergGroup$ with $\R^3$ via the natural map $(x,y,u)\in\R^3\mapsto \begin{pmatrix} 1 & x & u \\ 0 & 1 & y \\ 0 & 0 & 1 \end{pmatrix}$, which is a Lie group isomorphism once we equip $\R^3$ with the multiplication:
\begin{equation*}
(x_1,y_1,u_1)(x_2,y_2,u_2) = (x_1 + x_2, y_1 + y_2, u_1 + u_2 + x_1 y_2)
\end{equation*}
for all $(x_1,y_1,u_1), (x_2,y_2,u_2) \in \R^3$.

The importance of the Heisenberg group for quantum mechanics \cite{Folland89} may be gleaned by looking at its Lie algebra, which is given by:
\begin{equation*}
\mathds{h} = \left\{ \begin{pmatrix} 0 & x & u \\ 0 & 0 & y \\ 0 & 0 & 0 \end{pmatrix} : x,y,u \in \R \right\}
\end{equation*}
which is a $2$-nilpotent Lie algebra. We easily compute that for all $x,y,u \in \R^3$:
\begin{equation}\label{Heisenberg-Exp-eq}
\exp\begin{pmatrix}
0 & x & u \\
0 & 0 & y \\
0 & 0 & 0
\end{pmatrix}
= \begin{pmatrix}
1 & x & u + \frac{1}{2} xy \\ 0 & 1 & y \\ 0 & 0 & 1
\end{pmatrix}\text{.}
\end{equation}
This expression for the exponential will be important for our construction. Note that the exponential map is both injective and surjective.

We now set:
\begin{equation*}
P = \begin{pmatrix} 0 & 1 & 0 \\ 0 & 0 & 0 \\ 0 & 0 & 0 \end{pmatrix}\text{, }
Q = \begin{pmatrix} 0 & 0 & 0 \\ 0 & 0 & 1 \\ 0 & 0 & 0 \end{pmatrix}\text{ and }
T = \begin{pmatrix} 0 & 0 & 1 \\ 0 & 0 & 0 \\ 0 & 0 & 0 \end{pmatrix}\text{. }
\end{equation*}
We easily check that $[P,Q] = T = -[Q,P]$ while other other commutators between $P$, $Q$ and $T$ are null, and $\mathrm{span}_\C \{P,Q,T\} = \mathds{h}$. 

We note that in particular, $T$ is central, and thus the relations defining $\mathds{h}$ from the basis $\{P,Q,T\}$ are the structural equations of quantum mechanics --- the canonical commutation relation, as proposed by Heisenberg, in order to express the uncertainty principle between two conjugate observables. We refer to \cite{Folland89} for a detailed analysis of the Heisenberg group and its connections to the Moyal product, pseudo-differential calculus, and more fascinating topics.

Thus the study of the irreducible representations of $\HeisenbergGroup$ provide the irreducible representations of the canonical commutation relations. We first note that:
\begin{equation*}\
\bigslant{\HeisenbergGroup}{\{(0,0,u) : u \in \R\}} = \R^2
\end{equation*}
is Abelian, and thus we get a collection of trivial, one-dimensional representations of $\HeisenbergGroup$ by simply lifting the irreducible representations of $\R^2$.

If we set, for any $\eth\in\R\setminus\{0\}$ and $\xi \in L^2(\R)$:
\begin{equation}\label{alpha-action-eq}
\alpha_{\eth,1}^{x,y,u} \xi : s\in \R \mapsto \exp(2i\pi(\eth u + s x))\xi(s + \eth y)
\end{equation}
then we define a unitary representation of $\HeisenbergGroup$, and any nontrivial irreducible unitary representations of the Heisenberg group is unitarily equivalent to $\alpha_{\eth,1}$ for some $\eth\not=0$ \cite{Folland89}. We note that they all are infinite dimensional (the other, trivial, unitary representations of $\HeisenbergGroup$ are one-dimensional).

Let $\eth \in \R\setminus\{0\}$. For all $(x,y) \in \R^2$ and for all $\xi \in L^2(\R)$, set:
\begin{equation*}
\begin{split}
\sigma_{\eth,1}^{x,y}\xi &= \alpha_{\eth,1}^{\exp_{\mathds{H}_3}(xP + yQ)} \xi\\
&= \alpha_{\eth,1}^{x,y,\frac{x y}{2}}\xi : s \in \R \mapsto \exp(i\pi\eth x y + 2i\pi s x)\xi(s + \eth y)  \text{.}
\end{split}
\end{equation*}

The map $\sigma_{\eth,1}^{x,y}$ is a unitary on $L^2(\R)$ for all $(x,y) \in \R^2$. Moreover, for all $(x_1,y_1)$, $(x_2,y_2) \in \R^2$, we note that:
\begin{equation*}
\sigma_{\eth,1}^{x_1,y_1}\sigma_{\eth,1}^{x_2,y_2} = \cocycle{\eth}((x_1,y_1),(x_2,y_2)) \sigma_{\eth,1}^{x_1 + x_2, y_1 + y_2} \text{,}
\end{equation*}
i.e. $\sigma_{\eth,1}$ is a projective representation of $\R^2$ on $L^2(\R)$ for the bicharacter $\cocycle{\eth}$, namely the \emph{Schr{\"o}dinger representation} of ``Plank constant'' $\eth$. Moreover, every nontrivial irreducible unitary projective representation of $\R^2$ is unitarily equivalent to one of $\sigma_{1,\eth}$ for some $\eth \not= 0$ (by nontrivial, we mean associated with a nontrivial cocycle).

We introduce one more notation which will prove very useful in defining our D-norm on Heisenberg modules. If $d\in\N$ with $d>0$, we define the following unitarry operators on $L^2(\R)\otimes\C^d$:
\begin{equation*}
\alpha_{\eth,d}^{x,y,u} = \alpha_{\eth,1}^{x,y,u}\otimes\mathrm{id}\text{ and }\sigma_{\eth,d}^{x,y} = \sigma_{\eth,1}^{x,y}\otimes \mathrm{id}
\end{equation*} 
for all $x,y,u\in\R$, where $\mathrm{id}$ is the identity map on $\C^d$. We trivially check that $\alpha_{\eth,d}$ is a unitary representation of $\HeisenbergGroup$ on $L^2(\R)\otimes\C^d$, while $\sigma_{\eth,d}$ is a $\cocycle{\eth}$-projective representation of $\R^2$ on $L^2(\R)\otimes\C^d$. Moreover, we also check immediately that $\alpha_{\eth,d}^{x,y,0} = \sigma_{\eth,d}^{x,y}$ for all $x,y \in \R$. 
\end{notation}

We now turn to the projective representations of $\Z_q^2$, where $q\in\N\setminus\{0\}$. We first note that, for any $p\in\Z$, the skew bicharacter $\cocycle{\frac{p}{q}}$ of $\Z^2$ induces a skew bicharacter of $\Z_q^2$ --- which we keep denoting by $\cocycle{\frac{p}{q}}$. By \cite{Kleppner65}, any multiplier of $\Z_q^2$ is cohomologous to $\cocycle{\frac{p}{q}}$ for some $p\in\N$. 

For our purpose, we will thus get, up to unitary equivalence, every possible finite dimensional unitary projective representations of the groups $\Z_q^2$ for arbitrary $q\in\N\setminus\{0\}$ by considering the following family.

\begin{notation}\label{zq-Heisenberg-rep-notation}
Let $p\in\Z$ and $q\in\N\setminus\{0\}$. Let $n\in\Z\mapsto [n] \in \Z_q$ be the canonical surjection. Let:
\begin{equation*}
u_{p,q} = \begin{pmatrix}
1 &      & & & &\\
  & z & & & &\\
  & & z^2 & & &\\
  & & & \ddots & & \\
  & & & & & z^{q-1}
\end{pmatrix}
\text{ and }
v_{p,q} = \begin{pmatrix}
0 & \hdots &        & & &1 \\
1 & 0      & \hdots & & &  \\
\ddots & \ddots &   & & &  \\
  &        &        & &1& 0\\ 
\end{pmatrix}
\text{,}
\end{equation*}
with $z = \exp\left(\frac{2i\pi p}{q}\right)$. Since $u_{p,q}^q = v_{p,q}^q = \unit_{}$, the map:
\begin{equation*}
  \rho_{p,q,1} : (z,w) \in \Z_q^2 \mapsto \rho_{p,q,1}^{z,w} = \exp\left(\frac{i\pi p n m}{q}\right)u_{p,q}^n v_{p,q}^m \text{ where $[n] = z$ and $[m]=w$}
\end{equation*}
is well-defined. An easy computation shows that $\rho_{p,q,1}$ is a projective representation of $\Z_q^2$. 

For all $d\in q\N$, $d > 0$, we now set:
\begin{equation*}
\rho_{p,q,d}^{n,m} = \rho_{p,q,1}^{n,m}\otimes\mathrm{id}_{\frac{d}{q}}
\end{equation*}
where $\mathrm{id}_{\frac{d}{q}}$ is the identity map on $\C^{\frac{d}{q}}$.

We remark that $\rho_{p,q,d}$ acts on $\C^d$, i.e. we parametrized $\rho$ by the dimension of the space on which it acts rather than the multiplicity of $\rho_{p,q,1}$, as it will make our notations much simpler.

If $p$ and $q$ are relatively prime, the representation $\rho_{p,q,1}$ is irreducible, with range the entire algebra of $q\times q$ matrices --- it is in fact, the only irreducible $\cocycle{\frac{p}{q}}$-projective representation of $\Z_q^2$ up to unitary equivalence. Thus in general, any finite dimensional unitary representation of $\Z_q^2$ is unitarily equivalent to some $\rho_{l,r,d}$ for some $l \in \Z$, $r\in\N\setminus\{0\}$, $d\in r\N\setminus\{0\}$, with $l = 0$ and $r = 1$ or $l, r$ relatively prime.
\end{notation}

In order to construct the inner product on the Heisenberg modules, we shall need to first work on a space of well-behaved functions inside the Hilbert space $\ell^2(\Z^2)$ on which quantum tori will act. This space will consist of the Schwarz functions.

\begin{definition}
Let $E$ be a finite dimensional vector space. A function $f : \R \rightarrow E$ is a \emph{$E$-valued Schwarz function over $\R$} when it is infinitely differentiable on $\R$ and, for all $j \in \N$ and all polynomial $p \in \R[X]$, we have:
\begin{equation*}
\lim_{t\rightarrow\pm\infty} \left\| p(t)f^{(j)}(t) \right\|_E = 0 \text{.}
\end{equation*}

The space of all $E$-valued Schwarz functions over $\R$ is denoted by $\mathcal{S}(E)$.
\end{definition}

We note that if $f \in \mathcal{S}(E)$ for some finite dimensional space $E$, then in particular, $f \in L^p(\R)$ for all $p \in [1,\infty]$, since for any $j\in\N$, there exists $M > 0$ such that $\|f(s)\|_E \leq\frac{M}{1 + |s|^j}$ for all $s\in\R$.

We now implement the scheme which we described a few paragraphs above to construct modules over quantum tori. We refer to the mentioned works of Connes and Rieffel for the details and justification behind the following construction.

\begin{theorem-definition}[{ \cite{Connes80}, \cite{Rieffel83}, \cite{ConnesRieffel87}  }]\label{HeisenbergMod-thm}
Let $\theta \in \R$ and $q\in\N\setminus\{0\}$. Let $p \in \Z$, $q \in \N\setminus\{0\}$ , and let $d \in q\N\setminus\{0\}$. The \emph{Heisenberg module} $\HeisenbergMod{\theta}{p,q,d}$ is the module over $\qt{\theta}$ defined as follows.

Let $\rho_{p,q,d}$ be the projective action of $\Z_q^2$ with cocycle $\cocycle{\frac{p}{q}}$, consisting of the sum of $\frac{d}{q}$ copies of the unique, up to unitary equivalence, irreducible representation with the same cocycle. Up to unitary conjugation, we assume that $\rho_{p,q,d}$ acts on $\C^d$.

Let:
\begin{equation*}
\eth = \theta - \frac{p}{q}\text{.}
\end{equation*}

Let $\alpha_{\eth,1}$ be the action of the Heisenberg group $\HeisenbergGroup$ on $L^2(\R)$ given by Expression (\ref{alpha-action-eq}).

For $(n,m) \in \Z^2$, denoting the class of $n$ and $m$ in $\bigslant{\Z}{q\Z}$, respectively, by $[n]$ and $[m]$, we set:
\begin{equation*}
\varpi_{p,q,\eth,d}^{n,m} = \sigma_{\eth,1}^{n,m}\otimes\rho_{p,q,d}^{[n],[m]} \text{.}
\end{equation*}

For all $n,m\in\Z$, the map $\varpi_{p,q,\eth,d}^{n,m}$ is a unitary of $L^2(\R) \otimes \C^d$, and moreover $\varpi_{p,q,\eth,d}$ is an $\cocycle{\theta}$-projective representation of $\Z^2$.

By universality, the Hilbert space $L^2(\R)\otimes \C^d$ is a module over $\qt{\theta}$, with, in particular, for all $f\in \ell^1(\Z^2)$ and $\xi \in L^2(\R,\C^d) = L^2(\R)\otimes \C^d$:
\begin{equation*}
f \xi = \sum_{n,m\in\Z} f(n,m) \varpi_{p,q,\eth,d}^{n,m}\xi\text{.}
\end{equation*}

Let $\module{S}_\theta^{p,q,d} = \mathcal{S}(\C^d) \subseteq L^2(\R)\otimes \C^d$. For all $\xi,\omega \in \module{S}_\theta^{p,q,d}$, define $\inner{\xi}{\omega}{\HeisenbergMod{\theta}{p,q,d}}$ as the function in $\ell^1(\Z^2)$ given by:
\begin{equation*}
\inner{\xi}{\omega}{\HeisenbergMod{\theta}{p,q,d}} : (n,m) \in \Z^2 \longmapsto \inner{\varpi_{p,q,\eth,d}^{n,m}\xi}{\omega}{L^2(\R)\otimes E} \text{.}
\end{equation*}

The \emph{Heisenberg module $\HeisenbergMod{\theta}{p,q,d}$} is the completion of $\module{S}_\theta^{p,q,d}$ for the norm associated with the $\qt{\theta}$-inner product $\inner{\cdot}{\cdot}{\HeisenbergMod{\theta}{p,q,d}}$.
\end{theorem-definition}

We note that $\module{S}_\theta^{p,q,d}$ is not closed under the action of $\qt{\theta}$ but it is closed under the action of the subalgebra:
\begin{equation*}
\{ f \in \ell^1(\Z^2) : \forall p\in \R[X,Y] \quad \lim_{n,m\rightarrow\pm\infty} p(n,m)f(n,m) = 0 \}
\end{equation*}
of $(\ell^1(\Z^2),\conv{\theta},\cdot^\ast)$, often referred to as the smooth quantum torus. We will not use this observation later on, though it is notable that the completion of $\module{S}_\theta^{p,q,d}$ is indeed a $\qt{\theta}$-module.

\section{A continuous fields of $C^\ast$-Hilbert norms}

All Heisenberg modules are completions of $\mathcal{S}(\C^d)$ for some $d\in \N$, $d>0$. For a fixed $d$, it thus becomes possible to ask whether the various $C^\ast$-Hilbert norms $\|\cdot\|_{\HeisenbergMod{\theta}{p,q,d}}$, as $\theta$ varies in $\R$, form a continuous family. 

To this end, we establish a succession of lemmas whose primary goal is to provide us with estimates on the Heisenberg modules' $C^\ast$-Hilbert norms in terms of the norm of $\ell^1(\Z^2)$. While the Heisenberg modules' $C^\ast$-Hilbert norms are in general delicate to work with as they involve the no-less abstract quantum tori norms, the $\ell^1(\Z^2)$ norm, which dominates all of the quantum tori norms, is much more amenable to computations. For our purpose, we will take full advantage of the regularity of Schwarz functions, which will enable us to apply various analytic tools to derive our desired result.

The first step is a lemma which provides a first upper bound to the $\ell^1(\Z^2)$ norm of the difference between certain Heisenberg module inner products. 

\begin{lemma}\label{ip-lemma}
If $\theta,\vartheta\in\R$ and $p\in\Z$, $q \in \N\setminus\{0\}$, $d\in q\N\setminus\{0\}$, and if $\omega$, $\eta$ and $\xi$ are $C^2$ functions from $\R$ to $\C^d$ such that for all $f \in \{\omega,\eta,\xi\}$:
\begin{enumerate}
\item all of $f$, $f'$ and $f''$ are integrable on $\R$,
\item $\lim_{t\rightarrow\pm\infty} f(t) = \lim_{t\rightarrow\pm\infty} f'(t) = \lim_{t\rightarrow\pm\infty} f''(t) = 0$,
\end{enumerate}
then, writing $\eth_\theta = \theta - \frac{p}{q}$ and $\eth_\vartheta = \vartheta - \frac{p}{q}$, we have:
\begin{multline*}
\left\| \inner{\omega}{\eta}{\HeisenbergMod{\theta}{p,q,d}} - \inner{\xi}{\eta}{\HeisenbergMod{\vartheta}{p,q,d}} \right\|_{\ell^1(\Z^2)} \\
\leq \sum_{n\in\Z}\frac{1}{4\pi^2 n^2} \left( \int_\R \sum_{m\in\Z} \left\|\omega''(t + \eth_\theta m) - \xi''(t + \eth_\vartheta m)\right\|_{\C^d} \|\eta(t)\|_{\C^d} \, dt + \right. \\
+ 2\int_\R \sum_{m\in\Z} \left\|\omega'(t + \eth_\theta m)-\xi'(t + \eth_\vartheta m)\right\|_{\C^d} \|\eta'(t)\|_{\C^d} \, dt \\
+ \left. \int_\R \sum_{m\in\Z} \left\|\omega(t + \eth_\theta m) - \xi(t + \eth_\vartheta m)\right\|_{\C^d} \|\eta''(t)\|_{\C^d} \, dt \right) \text{.}
\end{multline*}
\end{lemma}

\begin{proof}
We begin with the observation that for all $(n,m) \in \Z^2$ we have:
\begin{multline*}
\inner{\omega}{\eta}{\HeisenbergMod{\theta}{p,q,d}}(n,m) - \inner{\xi}{\eta}{\HeisenbergMod{\vartheta}{p,q,d}}(n,m) \\
\begin{split}
&= \int_\R \inner{\rho_{p,q,d}^{[n],[m]} \omega(t + \eth_\theta m)}{\eta(t)}{\C^d} \exp(2i\pi n t)\,dt \\
&\quad -  \int_\R \inner{\rho_{p,q,d}^{[n],[m]} \xi(t + \eth_\vartheta m)}{\eta(t)}{\C^d} \exp(2i\pi n t)\,dt \\
&= \int_\R \inner{\rho_{p,q,d}^{[n],[m]} \left(\omega(t + \eth_\theta m)-\xi(t + \eth_\vartheta m)\right)}{\eta(t)}{\C^d} \exp(2i\pi n t) \,dt \text{.}
\end{split}
\end{multline*}

For all $n,m\in\Z$, the function:
\begin{equation*}
f_{n,m} : t \mapsto \inner{\rho_{p,q,d}^{[n],[m]} \omega(t + \eth_\theta m) - \xi(t + \eth_\vartheta m)}{\eta(t)}{\C^d}
\end{equation*}
has a first and continuous second derivative which are integrable, and:
\begin{equation*}
\lim_{t\rightarrow\pm\infty} f_{n,m}(t) = \lim_{t\rightarrow\pm\infty} f_{n,m}'(t) = \lim_{t\rightarrow\pm\infty} f_{n,m}''(t) = 0\text{.}
\end{equation*}
We consequently may apply integration by part and obtain, for all $m,n \in \Z$:
\begin{multline*}
\int_\R \inner{\rho_{p,q,d}^{[n],[m]} \omega(t + \eth_\theta m) - \xi(t + \eth_\vartheta m)}{\eta(t)}{\C^d} \exp(2i\pi n t) \, dt \\
\begin{split}
&= \int_\R f_{n,m}(t) \exp(2i\pi n t) \,dt = -\int_\R f_{n,m}'(t) \frac{\exp(2i\pi n t)}{2i\pi n} \, dt \\
&= \int_\R f_{n,m}''(t) \frac{\exp(2i\pi n t)}{4\pi^2 n^2} \, dt\text{.}
\end{split}
\end{multline*}

We compute trivially that for all $t\in\R$ and $m,n\in\Z$:
\begin{multline*}
f_{n,m}''(t) = \inner{\rho_{p,q,d}^{[n],[m]} \left(\omega''(t + \eth_\theta m) - \xi''(t + \eth_\vartheta m)\right)}{\eta(t)}{\C^d}\\
+ 2\inner{\rho_{p,q,d}^{[n],[m]} \left(\omega'(t + \eth_\theta m)-\xi'(t + \eth_\vartheta m)\right)}{\eta'(t)}{\C^d} \\ + \inner{\rho_{p,q,d}^{[n],[m]} \left(\omega(t + \eth_\theta m)-\xi(t + \eth_\vartheta m)\right)}{\eta''(t)}{\C^d} \text{.}
\end{multline*}

Thus using Cauchy-Schwarz and since $\rho_{p,q,d}^{[n],[m]}$ is a unitary, we conclude:
\begin{multline*} 
\left\| \inner{\omega}{\eta}{\HeisenbergMod{\theta}{p,q,d}} - \inner{\xi}{\eta}{\HeisenbergMod{\vartheta}{p,q,d}}\right\|_{\ell^1(\Z^2)} \\
\begin{split}
&= \sum_{n,m\in\Z} \left| \int_\R \inner{\rho_{p,q,d}^{[n],[m]} \left(\omega(t + \eth_\theta m) - \xi(t + \eth_\vartheta m)\right)}{\eta(t)}{\C^d}\exp(2i\pi n t) \, dt \right|\\
&\leq \sum_{m,n\in\Z} \int_\R \frac{\left| f_{n,m}''(t) \right| }{4\pi^2 n^2} \,dt \\
&\leq \sum_{m,n\in\Z} \frac{1}{4 \pi^2 n^2}\left(\int_\R \|\omega''(t + \eth_\theta m)-\xi''(t + \eth_\vartheta m)\|_{\C^d}\|\eta(t)\|_{\C^d} \,dt \right. \\
&\quad + 2\int_\R \|\omega'(t + \eth_\theta m)-\xi'(t + \eth_\vartheta m)\|_{\C^d}\|\eta'(t)\|_{\C^d} \,dt \\
&\quad \left. + \int_\R \|\omega(t + \eth_\theta m)-\xi(t + \eth_\vartheta m)\|_{\C^d}\|\eta''(t)\|_{\C^d} \,dt\right)\\
&= \sum_{n\in\N}\frac{1}{4\pi^2 n^2} \left[ \int_\R \left(\sum_{m\in\N}\|\omega''(t + \eth_\theta m)-\xi''(t + \eth_\vartheta m)\|_{\C^d}\right) \|\eta(t)\|_{\C^d} \,dt \right.\\
&\quad + 2\int_\R \left(\sum_{m\in\N}\|\omega'(t + \eth_\theta m)-\xi'(t + \eth_\vartheta m)\|_{\C^d}\right) \|\eta'(t)\|_{\C^d} \,dt\\
&\quad \left. + \int_\R \left(\sum_{m\in\N}\|\omega(t + \eth_\theta m)-\xi(t + \eth_\vartheta m)\|_{\C^d}\right) \|\eta''(t)\|_{\C^d} \,dt \right] \text{ by Tonelli's theorem.}
\end{split}
\end{multline*}
This concludes our lemma.
\end{proof}

Our next lemma focuses on the type of estimates given in Lemma (\ref{ip-lemma}), and gives a sufficient condition for these upper bounds to converge to $0$ when various parameters are allowed to converge to appropriate values. 

\begin{lemma}\label{ldom-lemma}
Let $d\in \N$, $d > 0$. Let $\Nbar = \N\cup\{\infty\}$ be the one point compactification of $\N$.

If $(\omega_k)_{k \in \Nbar}$ and $(\eta_k)_{k\in\Nbar}$ are two families of $C^2$-functions from $\R$ to $\C^d$ and $(\eth_k)_{k\in\N}$ is a sequence of nonzero real numbers converging to some $\eth_\infty \not= 0$ such that:
\begin{enumerate}
\item $(t,k) \in \R \times \Nbar \mapsto \omega_k(t)$ and $(t,k)\in\R\times\Nbar \mapsto \eta_k(t)$ are jointly continuous,
\item there exists $M > 0$ such that for all $k \in \Nbar$ and $t\in \R$:
\begin{equation*}
\max\left\{ \|\omega_k(t)\|_{\C^d}, \|\eta_k(t)\|_{\C^d} \right\} \leq \frac{M}{1 + t^2} \text{,}
\end{equation*}
\end{enumerate}
then:
\begin{equation}\label{limit-eq}
\lim_{k\rightarrow\infty} \sum_{n\in\N} \frac{1}{4\pi^2 n^2} \int_\R \sum_{m\in\Z}\|\omega_k(t + \eth_k m) - \omega_\infty(t + \eth_\infty m)\|_{\C^d} \|\eta_k(t)\|_{\C^d} \, dt = 0\text{.}
\end{equation}
\end{lemma}

\begin{proof}
First, we observe that Expression (\ref{limit-eq}) is left unchanged if we replace $\eth_k$ with $-\eth_k$ for all $k\in\Nbar$, thanks to the summation over $m\in\Z$. Consequently, we may assume without loss of generality that $\eth_\infty > 0$ and assume that $\eth_k > 0$ for all $k\in\Nbar$ (since $(\eth_k)_{k\in\N}$ converges to $\eth_\infty\not= 0$, we must have that $\eth_k$ and $\eth_{\infty}$ have the same sign for $k$ larger than some $K \in \N$; we thus can truncate our sequence to start at $K$ and flip all the signs if necessary to work with positive values). 

With this in mind, since $(\eth_k)_{k\in\N}$ is positive and converges to $\eth_\infty > 0$, there exists $0 < \eth_- < \eth_+$ such that for all $k \in \Nbar$, we have $\eth_k \in [\eth_-,\eth_+]$.

We shall employ the Lebesgue dominated convergence theorem. To this end, we introduce the following function to serve as our upper bound. For all $t,m\in\R$ we set:


\begin{equation}\label{b-eq}
\mathrm{b}(t,m) = \begin{cases}
\frac{M}{1 + (t + m \eth_-)^2} \text{ if $m > 0$ and $t\geq - \eth_- m$, or if $m < 0$ and $t\leq -\eth_- m$,}\\ 
M \text{ otherwise.}
\end{cases}
\end{equation}

For a fixed $t \in \R$, we note that:
\begin{equation*}
\mathrm{b}(t,m) \sim_{m\rightarrow \pm \infty} \frac{M}{\eth_-^2 m^2}\text{,}
\end{equation*}
so $\sum_{m\in\Z} \mathrm{b}(t,m) < \infty$. Moreover, by construction, for all $t,m\in \R$ and $\eth\in [\eth_-, \eth_+]$, we have:
\begin{equation*}
\frac{M}{1 + (t + \eth m)^2} \leq \mathrm{b}(t,m)\text{.}
\end{equation*}

Therefore, using our hypothesis, for all $t\in\R$, $m\in\Z$, $k\in\Nbar$ and $\eth \in [\eth_-, \eth_+]$:
\begin{equation*}
\begin{split}
\left\|\omega_k(t + m\eth)-\omega_\infty(t + m\eth_\infty)\right\|_{\C^d} &\leq \frac{M}{1 + (t + m\eth)^2} + \frac{M}{1 + (t + m \eth_\infty)^2} \\
&\leq 2\mathrm{b}(t,m) \text{.}
\end{split}
\end{equation*}

Thus for a fixed $t\in \R$, we may apply Lebesgue dominated convergence theorem to conclude:
\begin{equation}\label{ldom-eq1}
\lim_{k\rightarrow\infty}\sum_{m\in\Z}\left\|\omega_k(t + m \eth_k) - \omega_\infty(t + m \eth_\infty)\right\|_{\C^d} = 0\text{,}
\end{equation}
since $(t,k)\in\R\times\Nbar \mapsto \omega_k(t)$ is jointly continuous.

We now make another observation. For any fixed $\eth > 0$ and $k\in\Nbar$, The function:
\begin{equation*}
t\in\R \mapsto \sum_{m\in\Z} \|\omega_k(t + \eth m)\|_{\C^d}
\end{equation*}
is $\eth$-periodic.

If $t \in [0,\eth_+]$, $k\in\Nbar$ and $\eth \in [\eth_-,\eth_+]$, then since:
\begin{equation*}
\|\omega_k(t + \eth m)\|_{\C^d} \leq \sup_{x\in[0,\eth_+]}\mathrm{b}(x,m)
\end{equation*}
while, as can easily be checked:
\begin{equation*}
\sup_{x\in[0,\eth_+]} \mathrm{b}(x,m) \sim_{m\rightarrow\pm \infty} \frac{M}{\eth_-^2 m^2}\text{,}
\end{equation*}
we conclude that the series:
\begin{equation*}
\left( (t,k,\eth)\in\R\times\Nbar\times[\eth_-,\eth_+] \mapsto \sum\|\omega_k(t + \eth m)\|_{\C^d}\right)_{m\in\Z}
\end{equation*}
converges uniformly to its limit on $[0,\eth_+]\times \Nbar \times [\eth_-,\eth_+]$. In particular:
\begin{equation*}
(t,k,\eth) \in [0,\eth_+]\times\Nbar\times[\eth_-,\eth_+]\mapsto\sum_{m\in\Z}\|\omega_k(t + \eth m)\|_{\C^d}
\end{equation*}
is continuous on a compact domain and so it is bounded. Let $C > 0$ such that for all $(t,k,\eth) \in [0,\eth_+]\times\Nbar\times[\eth_-,\eth_+]$, we have:
\begin{equation*}
\sum_{m\in\Z}\|\omega_k(t + \eth m)\|_{\C^d} \leq C\text{.}
\end{equation*}

We conclude that $t\mapsto \sum_{m\in\Z}\|\omega_k(t-\eth_k m)\|_{\C^d}$ is bounded by $C$ on $\R$, since it is an $\eth_k$-periodic function with $\eth_k\leq \eth_+$, for all $k\in\Nbar$.
 
We thus have that for all $t \in \R$ and $k \in\Nbar$:
\begin{equation}\label{ldom-eq2}
\sum_{m\in\Z} \|\omega_k(t + m \eth_k) - \omega_\infty(t + m \eth_\infty)\|_{\C^d} \|\eta_k(t)\|_{\C^d} \leq 2 C \|\eta_k(t)\|_{\C^d} \leq \frac{2CM}{1 + t^2}\text{.}
\end{equation}
Now $t\in\R\mapsto\frac{2CM}{1 + t^2}$ is integrable over $\R$. Once again, we apply Lebesgue dominated convergence theorem, and we conclude from Expression (\ref{ldom-eq1}) that:
\begin{equation}\label{ldom-eq3}
\lim_{k\rightarrow \infty} \int_\R \sum_{m\in\Z}\|\omega_k(t + m \eth_k) - \omega_\infty(t + m \eth_\infty)\|_{\C^d} \|\eta_k(t)\|_{\C^d} \, dt = 0\text{.}
\end{equation}

Last, using Inequality (\ref{ldom-eq2}) again, we note that for all $k\in\Nbar$:
\begin{equation*}
\int_\R \sum_{m\in\Z}\|\omega_k(t + m \eth_k) - \omega_\infty(t + m \eth_\infty)\|_{\C^d} \|\eta_k(t)\|_{\C^d} \, dt \leq \int_\R \frac{2CM}{1 + t^2} \,dt = 2CM \pi
\end{equation*}
and thus for all $n\in\Z$ and $k\in\Nbar$:
\begin{equation*}
\frac{1}{4\pi^2 n^2} \int_\R \sum_{m\in\Z}\|\omega(t + m \eth_k) - \omega(t + m \eth_\infty)\|_{\C^d} \|\eta(t)\|_{\C^d} \, dt \leq \frac{2CM \pi}{4\pi^2 n^2} = \frac{CM}{2\pi n^2}\text{,} 
\end{equation*}
with $\sum_{n\in\Z} \frac{CM}{2\pi n^2} < \infty$; hence we may apply Lebesgue dominated convergence theorem once more to conclude from Expression (\ref{ldom-eq3}):
\begin{equation*}
\lim_{k\rightarrow\infty} \sum_{n\in\Z} \frac{1}{4\pi^2 n^2} \int_\R \sum_{m\in\Z}\|\omega_k(t + m \eth_k) - \omega_\infty(t + m \eth_\infty)\|_{\C^d} \|\eta_k(t)\|_{\C^d} \, dt = 0 \text{.}
\end{equation*}
This concludes our lemma.
\end{proof}

\begin{remark}\label{l1-rmk}
One may check that Lemma (\ref{ip-lemma}) and Lemma (\ref{ldom-lemma}) together prove that if $p,q \in \N$, $\xi,\omega\in\mathcal{S}(\C^d)$, for any $d \in q\N$ with $d > 0$, and if $\theta\in\R\setminus\left\{\frac{p}{q}\right\}$, then $\inner{\xi}{\omega}{\HeisenbergMod{\theta}{p,q,d}} \in \ell^1(\Z^2)$. It is a well-known fact (indeed a basic fact for the very construction of Heisenberg modules) though maybe not apparent from Theorem-Definition (\ref{HeisenbergMod-thm}) without consulting such sources as \cite{Rieffel83}.
\end{remark}

We now bring together Lemma (\ref{ip-lemma}) and Lemma (\ref{ldom-lemma}) to obtain a first result of continuity on the Heisenberg module inner products, albeit using the $\ell^1(\Z^2)$ norm. This is the core result of this section, and it is phrased at a somewhat higher level of generality that what is needed for the proof of continuity of the family of Heisenberg $C^\ast$-Hilbert norms. Indeed, this level of generality will prove useful twice later in this paper: when proving that the Heisenberg group representations $\alpha_{\eth,d}$ define \emph{strongly continuous} actions on Heisenberg modules, and when establishing that our prospective D-norms on Heisenberg modules will also form a continuous family of norms in \cite{Latremoliere18a}.

\begin{lemma}\label{l1-cont-lemma}
Let $p,q\in\N$ with $q > 0$ and $d\in q\N$ with $d > 0$. If $(\xi_k)_{k\in\Nbar}$ is a family of $\C^d$-valued $C^2$-functions over $\R$ such that:
\begin{enumerate}
\item there exists $M > 0$ such that for all $k\in\Nbar$ and $t\in\R$:
\begin{equation*}
\max\left\{ \|\xi_k(t)\|_{\C^d}, \|\xi_k'(t)\|_{\C^d}, \|\xi_k''(t)\|_{\C^d} \right\} \leq \frac{M}{1 + t^2}\text{,}
\end{equation*}
\item $(t,k) \in \R\times\Nbar \mapsto \xi_k(t)$ is continuous,
\end{enumerate}
and if $(\theta_k)_{k\in\N}$ is a sequence converging to $\theta_\infty$ such that $\theta_k-\frac{p}{q} \not=0$ for all $k\in\Nbar$, then we have:
\begin{equation*}
\lim_{k\rightarrow\infty} \left\| \inner{\xi_k}{\xi_k}{\HeisenbergMod{\theta_k}{p,q,d}} - \inner{\xi_\infty}{\xi_\infty}{\HeisenbergMod{\theta_\infty}{p,q,d}} \right\|_{\ell^1(\Z^2)}  = 0 \text{.}
\end{equation*}
\end{lemma}

\begin{proof}
To fix notations, for all $k\in\Nbar$, we set $\eth_k = \theta_k - \frac{p}{q}$. Note that $(\eth_k)_{k\in\N}$ is a sequence of nonzero real numbers converging to $\eth_\infty \not= 0$.

We shall prove our result from the following inequality:
\begin{multline}\label{l1-cont-eq1}
\left\| \inner{\xi_k}{\xi_k}{\HeisenbergMod{\theta_k}{p,q,d}} - \inner{\xi_\infty}{\xi_\infty}{\HeisenbergMod{\theta_\infty}{p,q,d}} \right\|_{\ell^1(\Z^2)} \\ \leq \left\| \inner{\xi_k}{\xi_k}{\HeisenbergMod{\theta_k}{p,q,d}} - \inner{\xi_k}{\xi_\infty}{\HeisenbergMod{\theta_\infty}{p,q,d}} \right\|_{\ell^1(\Z^2)} + \left\| \inner{\xi_k}{\xi_\infty}{\HeisenbergMod{\theta_\infty}{p,q,d}} - \inner{\xi_\infty}{\xi_\infty}{\HeisenbergMod{\theta_\infty}{p,q,d}} \right\|_{\ell^1(\Z^2)} \text{.}
\end{multline}

We begin with the first term of the right hand side of Inequality (\ref{l1-cont-eq1}). We observe that:
\begin{equation*}
\left\| \inner{\xi_k}{\xi_k}{\HeisenbergMod{\theta_k}{p,q,d}} - \inner{\xi_k}{\xi_\infty}{\HeisenbergMod{\theta_\infty}{p,q,d}} \right\|_{\ell^1(\Z^2)} = \left\| \inner{\xi_k}{\xi_k}{\HeisenbergMod{\theta_k}{p,q,d}} - \inner{\xi_\infty}{\xi_k}{\HeisenbergMod{\theta_\infty}{p,q,d}} \right\|_{\ell^1(\Z^2)}\text{.}
\end{equation*}

By Lemma (\ref{ip-lemma}), we then have for all $k\in\N$:
\begin{multline}\label{l1-cont-eq2}
\left\| \inner{\xi_k}{\xi_k}{\HeisenbergMod{\theta_k}{p,q,d}} - \inner{\xi_\infty}{\xi_k}{\HeisenbergMod{\theta_\infty}{p,q,d}} \right\|_{\ell^1(\Z^2)} \\
\leq \sum_{n\in\Z}\frac{1}{4\pi^2 n^2} \left( \int_\R \sum_{m\in\Z} \left\|\xi_k''(t + \eth_k m) - \xi_\infty''(t + \eth_\infty m)\right\|_{\C^d} \|\xi_k(t)\|_{\C^d} \, dt + \right. \\
+ 2\int_\R \sum_{m\in\Z} \left\|\xi_k'(t + \eth_k m)-\xi_\infty'(t + \eth_\infty m)\right\|_{\C^d} \|\xi_k'(t)\|_{\C^d} \, dt \\
+ \left. \int_\R \sum_{m\in\Z} \left\|\xi_k(t + \eth_k m) - \xi_\infty(t + \eth_\infty m)\right\|_{\C^d} \|\xi_k''(t)\|_{\C^d} \, dt \right) \text{.}
\end{multline}

Our assumptions allow us to apply Lemma (\ref{ldom-lemma}) to each term in the right hand side of Inequality (\ref{l1-cont-eq2}) to conclude that:
\begin{equation*}
\lim_{k\rightarrow\infty} \left\| \inner{\xi_k}{\xi_k}{\HeisenbergMod{\theta_k}{p,q,d}} - \inner{\xi_k}{\xi_\infty}{\HeisenbergMod{\theta_\infty}{p,q,d}} \right\|_{\ell^1(\Z^2)} = 0\text{.}
\end{equation*}

We handle the second term of Inequality (\ref{l1-cont-eq1}) in a similar manner.

From Inequality (\ref{l1-cont-eq1}), our lemma is proven.
\end{proof}

We now conclude this section with the proof that indeed, Heisenberg $C^\ast$-Hilbert norms form continuous families of norms for a fixed projective representation of some $\Z_q^2$.

\begin{proposition}\label{module-norm-continuity-prop}
Let $p,q \in \N$ and $d\in q\N$ with $d > 0$. Let $(\xi)_{k\in\Nbar}$ be a family in $\mathcal{S}(\C^d)$ such that $(k,t)\in\Nbar\times\R \mapsto \xi_k(t)$ is (jointly) continuous and there exists $M > 0$ such that $\|\xi_k^{(s)}(t)\|_{\C^d} \leq \frac{M}{1 + t^2}$ for all $k\in\Nbar$, $t\in \R$ and $s\in\{0,1,2\}$.

If $(\theta_k)_{k\in\N}$ is a sequence in $\R$ converging to $\theta_\infty$ and such that $\theta_k - \frac{p}{q} = 0$ for all $k\in\Nbar$, then:
\begin{equation*}
\lim_{k\rightarrow\infty} \left\| \xi_k \right\|_{\HeisenbergMod{\theta_k}{p,q,d}} = \left\|\xi_\infty\right\|_{\HeisenbergMod{\theta_\infty}{p,q,d}} \text{.}
\end{equation*}
\end{proposition}

\begin{proof}
For each $k\in\N\cup\left\{\infty\right\}$, we set $\eth_k = \theta_k - \frac{p}{q} \not= 0$. 

We first compute:

\begin{equation}\label{cont-norm-eq1}
\begin{split}
\left| \|\xi_k\|^2_{\HeisenbergMod{\theta_k}{p,q,d}} - \|\xi_\infty\|^2_{\HeisenbergMod{\theta_\infty}{p,q,d}} \right| &= \left| \left\|\inner{\xi_k}{\xi_k}{\HeisenbergMod{\theta_k}{p,q,d}}\right\|_{\qt{\theta_k}} - \left\|\inner{\xi_\infty}{\xi_\infty}{\HeisenbergMod{\theta_\infty}{p,q,d}}\right\|_{\qt{\theta_\infty}} \right|\\
&\leq \left| \left\|\inner{\xi_k}{\xi_k}{\HeisenbergMod{\theta_k}{p,q,d}}\right\|_{\qt{\theta_k}} - \left\|\inner{\xi_\infty}{\xi_\infty}{\HeisenbergMod{\theta_\infty}{p,q,d}}\right\|_{\qt{\theta_k}}\right| \\
&\quad + \left|\left\|\inner{\xi_\infty}{\xi_\infty}{\HeisenbergMod{\theta_\infty}{p,q,d}}\right\|_{\qt{\theta_k}} - \left\|\inner{\xi_\infty}{\xi_\infty}{\HeisenbergMod{\theta_\infty}{p,q,d}}\right\|_{\qt{\theta_\infty}} \right|\\
&\leq \left\| \inner{\xi_k}{\xi_k}{\HeisenbergMod{\theta_k}{p,q,d}} - \inner{\xi_\infty}{\xi_\infty}{\HeisenbergMod{\theta_\infty}{p,q,d}}\right\|_{\qt{\theta_k}} \\
&\quad + \left|\left\|\inner{\xi_\infty}{\xi_\infty}{\HeisenbergMod{\theta_\infty}{p,q,d}}\right\|_{\qt{\theta_k}} - \left\|\inner{\xi_\infty}{\xi_\infty}{\HeisenbergMod{\theta_\infty}{p,q,d}}\right\|_{\qt{\theta_\infty}} \right| \\
&\leq \left\| \inner{\xi_k}{\xi_k}{\HeisenbergMod{\theta_k}{p,q,d}} - \inner{\xi_\infty}{\xi_\infty}{\HeisenbergMod{\theta_\infty}{p,q,d}}\right\|_{\ell^1(\Z^2)} \\
&\quad + \left|\left\|\inner{\xi_\infty}{\xi_\infty}{\HeisenbergMod{\theta_\infty}{p,q,d}}\right\|_{\qt{\theta_k}} - \left\|\inner{\xi_\infty}{\xi_\infty}{\HeisenbergMod{\theta_\infty}{p,q,d}}\right\|_{\qt{\theta_\infty}} \right| \text{.}
\end{split}
\end{equation}

We now apply Lemma (\ref{l1-cont-lemma}) to conclude that:
\begin{equation*}
\lim_{k\rightarrow\infty} \left\| \inner{\xi_k}{\xi_k}{\HeisenbergMod{\theta_k}{p,q,d}} - \inner{\xi_\infty}{\xi_\infty}{\HeisenbergMod{\theta_\infty}{p,q,d}}\right\|_{\ell^1(\Z^2)} = 0\text{.}
\end{equation*}

Now, for any $f\in\ell^1(\Z^2)$, the function $\theta\in\R\mapsto\|f\|_{\qt{\theta}}$ is continuous by \cite[Corollary 2.7]{Rieffel89}. Hence, using Remark (\ref{l1-rmk}):
\begin{equation*}
\lim_{k\rightarrow\infty}\left|\left\|\inner{\xi_\infty}{\xi_\infty}{\HeisenbergMod{\theta_\infty}{p,q,d}}\right\|_{\qt{\theta_k}} - \left\|\inner{\xi_\infty}{\xi_\infty}{\HeisenbergMod{\theta_\infty}{p,q,d}}\right\|_{\qt{\theta_\infty}} \right| = 0\text{.}
\end{equation*}
Thus, we conclude from Inequality (\ref{cont-norm-eq1}) that:
\begin{equation*}
\lim_{k\rightarrow\infty} \left\| \xi_k \right\|_{\HeisenbergMod{\theta_k}{p,q,d}}^2 = \left\|\xi_\infty\right\|_{\HeisenbergMod{\theta_\infty}{p,q,d}}^2
\end{equation*}
which, by continuity of the square root, proves our lemma.
\end{proof}

\begin{corollary}\label{module-norm-continuity-cor}
Let $p,q \in \N$ and $d\in q\N$ with $d > 0$. Let $\xi\in\mathcal{S}(\C^d)$. If $(\theta_k)_{k\in\N}$ is a sequence in $\R$ converging to $\theta_\infty$ and such that $\theta_k - \frac{p}{q} = 0$ for all $k\in\Nbar$, then:
\begin{equation*}
\lim_{k\rightarrow\infty} \left\| \xi \right\|_{\HeisenbergMod{\theta_k}{p,q,d}} = \left\|\xi\right\|_{\HeisenbergMod{\theta_\infty}{p,q,d}} \text{.}
\end{equation*}
\end{corollary}

\begin{proof}
We apply Proposition (\ref{module-norm-continuity-prop}) to the family $k \in \Nbar \mapsto \xi$. We note that since $\xi$ is a Schwarz function, our assumptions are met.
\end{proof}

\section{The action of the Heisenberg group on Heisenberg modules}

Our goal in this paper is to prove that Heisenberg modules may be endowed with a {\gQVB} structure over quantum $2$-tori using a D-norm built from a Lie group action and inspired by the construction of \cite{Rieffel98a}, albeit involving a projective action of a locally compact group, which will not act via isometries of the D-norm. These changes will introduce new difficulties which we will handle in the next few sections. As a first step, we study the actions of the Heisenberg group on Heisenberg modules.

One motivation for the results in this section is to establish the properties which will meet the hypothesis of the main results in our next section, from which our D-norm will emerge. We also note that the actions $\alpha_{\eth,d}$, for all $\eth \in \R\setminus\{0\}$ and $d\in \N\setminus\{0\}$, is a strongly continuous action by isometries of $L^2(\R)\otimes\C^d$, but we need these results to be proven for the Heisenberg $C^\ast$-Hilbert norms, which dominate the norm of $L^2(\R)\otimes\C^d$.

We shall use the same hypotheses for a series of lemmas and our main definition in this section, and thus we group them in the following.

\begin{hypothesis}\label{hyp-1}
Let $p\in\Z$, $q \in\N\setminus\{0\}$, and let $d\in q\N$ with $d > 0$. Let $\theta \in \R\setminus\left\{\frac{p}{q}\right\}$. We write $\eth = \theta - \frac{p}{q}$.

We shall employ the notations of Theorem-Definition (\ref{HeisenbergMod-thm}).
\end{hypothesis}

We begin with two lemmas which will prove that $\HeisenbergGroup$ acts via isometries of the norm of the Heisenberg modules on the subspace of Schwarz functions --- where we have an explicit formula for our inner product --- and thus can indeed be extended to the entire module.

\begin{lemma}\label{alpha-beta-lemma}
We assume Hypothesis (\ref{hyp-1}). For all $(x,y,u) \in \HeisenbergGroup$, if $z_1 = \exp\left(2i\pi \eth y\right)$ and $z_2 = \exp\left(-2i\pi \eth x\right)$, and if $\xi,\omega \in \module{S}_\theta^{p,q,d}$, then:
\begin{equation*}
\inner{\alpha_{\eth,d}^{x,y,u}(\xi)}{\alpha_{\eth,d}^{x,y,u}(\omega)}{\HeisenbergMod{\theta}{p,q,d}} = \beta_\theta^{z_1,z_2}\left(\inner{\xi}{\omega}{\HeisenbergMod{\theta}{p,q,d}}\right) \text{.}
\end{equation*}
\end{lemma}

\begin{proof}
Let $n,m \in \Z$. We compute:

\begin{multline*}
\inner{\alpha_{\eth,d}^{x,y,u}(\xi)}{\alpha_{\eth,d}^{x,y,u}(\omega)}{\HeisenbergMod{\theta}{p,q,d}}(n,m) \\
\begin{split}
&= \inner{\varpi_{p,q,\eth,d}^{n,m}\alpha_{\eth,d}^{x,y,u} \xi}{\alpha_{\eth,d}^{x,y,u}\omega}{L^2(\R)\otimes\C^d}\\
&= \inner{ \left(\sigma_{\eth,1}^{n,m}\alpha_{\eth,1}^{x,y,u} \otimes \rho_{p,q,d}^{[n],[m]}\right) \xi}{\alpha_{\eth,d}^{x,y,u}\omega}{L^2(\R)\otimes\C^d} \\
&= \inner{\left(\alpha_{\eth,1}^{(x,y,u)^{-1}}\alpha_{\eth,1}^{n,m,\frac{nm}{2}}\alpha_{\eth,1}^{x,y,u} \otimes \rho_{p,q,d}^{[n],[m]}\right) \xi}{\omega}{L^2(\R)\otimes\C^d}\\
&= \inner{\exp(2i\pi\eth(y n - x m)) \left(\sigma_{\eth,1}^{n,m}\otimes\rho_{p,q,d}^{[n],[m]}\right) \xi}{\omega}{L^2(\R)\otimes\C^d} \\
&= z_1^n z_2^m \inner{\varpi_{p,q,\eth,d}^{n,m}\xi}{\omega}{L^2(\R)\otimes\C^d} \text{.}
\end{split}
\end{multline*}

Therefore, by definition of the dual action $\beta$:
\begin{equation*}
\inner{\alpha_{\eth,d}^{x,y,u}(\xi)}{\alpha_{\eth,d}^{x,y,u}(\omega)}{\HeisenbergMod{\theta}{p,q,d}} = \beta_\theta^{z_1,z_2}\left(\inner{\xi}{\omega}{\HeisenbergMod{\theta}{p,q,d}}\right)
\end{equation*}
as desired. %
\end{proof}

To ease our notations in this section, we set:
\begin{notation}
For all $(x,y) \in \R^2$ and $\eth > 0$, we define:
\begin{equation*}
\upsilon_\eth(x,y) = \left(\exp(2i\pi \eth y), \exp(-2i\pi \eth x) \right) \in \T^2 \text{.}
\end{equation*}
\end{notation}

We now show that the Heisenberg group acts by isometries for the $C^\ast$-Hilbert norm.

\begin{lemma}\label{alpha-isometry-lemma}
We assume Hypothesis (\ref{hyp-1}). For all $(x,y,u) \in \HeisenbergGroup$, the map $\alpha_{\eth,d}^{x,y,u}$ is an isometry of $\left(\HeisenbergMod{\theta}{p,q,d},\|\cdot\|_{\HeisenbergMod{\theta}{p,q,d}}\right)$.
\end{lemma}

\begin{proof}
Let $(x,y,u) \in \HeisenbergGroup$ and $\xi \in \module{S}_{\theta}^{p,q,d}$. We compute:
\begin{equation*}
\begin{split}
\left\|\alpha_\eth^{x,y,u}\xi\right\|_{\HeisenbergMod{\theta}{p,q,d}}^2 &= \left\|\inner{\alpha_{\eth,d}^{x,y,u}\xi}{\alpha_{\eth,d}^{x,y,u}\xi}{\HeisenbergMod{\theta}{p,q,d}}\right\|_{\qt{\theta}} \\
&= \left\|\beta_\theta^{\upsilon_r(x,y)}\inner{\xi}{\xi}{\HeisenbergMod{\theta}{p,q,d}}\right\|_{\qt{\theta}} \text{ by Lemma (\ref{alpha-beta-lemma}),}\\
&= \left\|\inner{\xi}{\xi}{\HeisenbergMod{\theta}{p,q,d}}\right\|_{\qt{\theta}} \\
&= \left\|\xi\right\|_{\HeisenbergMod{\theta}{p,q,d}}^2 \text{.}
\end{split}
\end{equation*}
This completes our proof.
\end{proof}

\begin{notation}
We use the notations of Hypothesis (\ref{hyp-1}). The action $\alpha_{\eth,d}$ of $\HeisenbergGroup$ on $\module{S}_\theta^{p,q,d}$ may thus be extended to $\HeisenbergMod{\theta}{p,q,d}$ by extending by continuity $\alpha_{\eth,d}^{x,y,u}$ for all $(x,y,u) \in \HeisenbergGroup$; we shall keep the notation of this extension as $\alpha_{\eth,d}$. We note that it also acts via isometry on  $\left(\HeisenbergMod{\theta}{p,q,d},\|\cdot\|_{\HeisenbergMod{\theta}{p,q,d}}\right)$.

We also use the same notation for $\sigma_{\eth,d}$ extended to $\left(\HeisenbergMod{\theta}{p,q,d},\|\cdot\|_{\HeisenbergMod{\theta}{p,q,d}}\right)$.
\end{notation}

The actions of the Heisenberg group on Heisenberg modules is by morphism modules, in the sense of \cite[Definition 3.5]{Latremoliere16c}. This result will play a role in the proof that our D-norm satisfies the modular version of the Leibniz inequality.

\begin{lemma}\label{alpha-morphism-lemma}
We assume Hypothesis (\ref{hyp-1}).  For all $a\in\qt{\theta}$, $\xi \in \HeisenbergMod{\theta}{p,q,d}$ and $(x,y,u) \in \HeisenbergGroup$, then:
\begin{equation*}
\alpha_{\eth,d}^{x,y,u}\left(a\xi\right) = \beta_\theta^{\upsilon_\eth(x,y)}(a)\alpha_{\eth,d}^{x,y,u}(\xi)\text{.}
\end{equation*}
\end{lemma}

\begin{proof}

Let $n,m \in \Z$ and $\xi \in \module{S}_\theta^{p,q,d}$ and $f_{m,m} \in \ell^1(\Z^2)$ be defined by:
\begin{equation*}
f_{n,m} : (z,w) \in \Z^2 \longmapsto \begin{cases}
1 \text{ if $n=z$ and $m=w$,}\\
0 \text{ otherwise.}
\end{cases}
\end{equation*}
We compute:

\begin{equation*}
\begin{split}
\alpha_{\eth,d}^{x,y,u}(f_{n,m} \xi) &= \alpha_{\eth,d}^{x,y,u}\varpi_{p,q,\eth,d}^{n,m}\xi \\
&= \left(\alpha_{\eth,d}^{x,y,u}\alpha_{\eth,d}^{n,m,\frac{n m}{2}} \otimes \rho_{p,q,d}^{[n],[m]}\right) \xi \\
&= \exp(2i\pi\eth (y n - x m)) \left(\alpha_{\eth,d}^{n,m,\frac{n m}{2}} \alpha_{\eth,d}^{x,y,u} \otimes\rho_{p,q,\eth,d}^{[n],[m]}\right) \xi \\
&= \exp(2i\pi \eth(y n - x m)) \varpi_{p,q,\eth,d}^{n,m} \alpha_{\eth,d} ^{x,y,u}\xi \\
&= \beta_\theta^{\upsilon_\eth(x,y)}(f_{n,m}) \alpha_{\eth,d}^{x,y,u} \xi \text{.}
\end{split}
\end{equation*}

Since $\beta_\theta$ is an action by *-morphisms, we conclude that for all $a\in\qt{\theta}$:
\begin{equation}\label{alpha-morphism-eq1}
\alpha_{\eth,d}^{x,y,u}(a \xi) = \beta_\theta^{\upsilon_\eth(x,y)}(a) \alpha_{\eth,d}^{x,y,u}(\xi)
\end{equation}
as desired. The lemma is concluded by extending Equality (\ref{alpha-morphism-eq1}) to $\HeisenbergMod{\theta}{p,q,d}$ by continuity.
\end{proof}

An important corollary of Lemma (\ref{alpha-morphism-lemma}) is as follows:

\begin{corollary}\label{alpha-norm-cor}
We assume Hypothesis (\ref{hyp-1}).  For all $a\in\qt{\theta}$, $\xi \in \HeisenbergMod{\theta}{p,q,d}$ and $(x,y,u) \in \HeisenbergGroup$, we observe that:
\begin{equation*}
\left\|\alpha_{\eth,d}^{x,y,u}(a\xi)\right\|_{\HeisenbergMod{\theta}{p,q,d}} \leq \|a\|_{\qt{\theta}}\|\xi\|_{\HeisenbergMod{\theta}{p,q,d}} \text{.}
\end{equation*}
\end{corollary}

\begin{proof}
Let $a\in\qt{\theta}$, $\xi \in \HeisenbergMod{\theta}{p,q,d}$ and $(x,y,u) \in \HeisenbergGroup$. We compute:
\begin{equation*}
\begin{split}
\left\|\alpha_{\eth,d}^{x,y,u}(a\xi)\right\|_{\HeisenbergMod{\theta}{p,q,d}} &= \left\|\beta_\theta^{\upsilon_\eth(x,y)} (a) \alpha_{\eth,d}^{x,y,u} \xi \right\|_{\HeisenbergMod{\theta}{p,q,d}} \text{ by Lemma (\ref{alpha-morphism-lemma}),}\\
&\leq \|\beta_\theta^{\upsilon_\eth(x,y)} a\|_{\qt{\theta}} \|\xi\|_{\HeisenbergMod{\theta}{p,q,d}} \text{ by Lemma (\ref{alpha-isometry-lemma}).}
\end{split}
\end{equation*}
This completes our proof.
\end{proof}

We have checked that the actions of the Heisenberg group on Heisenberg modules, which the latter were constructed from, act by isometric module morphisms on the entire module. Note that we already observed that Heisenberg modules can be regarded as dense subspaces of $L^2(\R)\otimes\C^d$ spaces on which the same action of the Heisenberg group is defined, strongly continuous and isometric; however we needed to ensure that these actions are well-behaved with respect to the inner product and norm of the Heisenberg modules. 

In order to define our D-norms, we shall require one more important analytic property: we want our actions to be strongly continuous for the Heisenberg $C^\ast$-Hilbert norms. This is the subject of the next proposition. We actually include in the next proposition a somewhat more general hypothesis and estimate than needed for the strong continuity of our actions, as this stronger statement will play an important role in our study of the continuity properties of our D-norms later on in \cite{Latremoliere18a}.

\begin{proposition}\label{strongly-continuous-prop}
Let $p\in\Z$, $q\in \N\setminus\{0\}$ and $d \in q\N$ with $d > 0$. Let $C > 0$ and $M > 0$ some constant. Let $0 < \eth_- < \eth_+$. There exists $K > 0$ such that for all $\xi \in \module{S}(\C^d)$ satisfying:
\begin{multline}\label{strongly-continuous-prop-eq0}
\max\left\{ \|\xi(s)\|_{\C^d}, \|s\xi(s)\|_{\C^d}, \|\xi'(s)\|_{\C^d}, \|s\xi'(s)\|_{\C^d}, \right.\\
\left. \|\xi''(s)\|_{\C^d}, \|s\xi''(s)\|_{\C^d} \right\} \leq \frac{M}{1 + s^2}\text{,}
\end{multline}
the following holds for all $s\in \R$, $\eth \in [\eth_-,\eth_+]$ and $(x,y,u)\in\R^3$ with $|x|+|y|+|u| \leq C$:
\begin{equation}\label{strongly-continuous-eq1}
\max\left\{ \left\| \alpha_{\eth,d}^{x,y,u}\xi^{(n)}(s) - \xi^{(n)} (s)\right\|_{\C^d} : n \in \{0,1,2\}  \right\} \leq \frac{K(|x|+|y|+|u|)}{1 + s^2} \text{.}
\end{equation}
In particular, for all $\eth \not= 0$ and $\theta = \eth + \frac{p}{q}$:
\begin{equation*}
\lim_{(x,y,u) \rightarrow 0} \left\| \alpha_{\eth,d}^{x,y,u}\xi - \xi\right\|_{\HeisenbergMod{\theta}{p,q,d}} = 0 \text{.}
\end{equation*}
\end{proposition}

\begin{proof}
Let $\xi \in \mathcal{S}(\C^d)$ and $(x,y,u) \in \R^3$. We note that for all $s\in\R$, using the continuity of $\xi$, we of course have:
\begin{equation*}
\begin{split}
\alpha_{\eth,d}^{x,y,u}\xi(s) - \xi (s) &= \exp(2i\pi (u + xs))\xi(s + \eth y) - \xi(s)\\
&\xrightarrow{(x,y,u)\rightarrow 0} 0 \text{.}
\end{split}
\end{equation*}

However, we wish to apply Lemma (\ref{l1-cont-lemma}) to obtain convergence in norm, so we seek a more precise estimate. To this end, let:
\begin{equation*}
f_s(t) = \alpha_{\eth,d}^{tx,ty,tu}\xi(s) = \exp(2i\pi(\eth tu + txs))\xi(s + \eth ty)
\end{equation*}
for all $t,s\in\R$. We compute for all $t,s\in\R$:
\begin{equation*}
\begin{split}
f_s'(t) = \exp(2i\pi(\eth tu + txs))\left(2i\pi (\eth u + xs) \xi(s + \eth ty) +  \eth y \xi'(s+\eth ty) \right) \text{.}
\end{split}
\end{equation*}

Let $\|(x,y,u)\|_1 = |x| + |y| + |u|$ for all $(x,y,u)\in\R^2$, i.e. $\|\cdot\|_1$ is the usual $1$-norm on $\R^3$. Let us now assume $\|(x,y,u)\|_1 \leq C$ --- in particular, $|y| < C$. We observe that for all $s\in \R$, using the function $\mathrm{b}$ introduced in Expression (\ref{b-eq}) in the proof of Lemma (\ref{ldom-lemma}):
\begin{multline*}
\left\| \alpha_{\eth,d}^{x,y,u}\xi(s) - \xi(s)\right\|_{\C^d} \\
\begin{split}
&= \left\| f_s(1) - f_s(0) \right\|_{\C^d} = \left\| \int_0^1 f_s'(t)\, dt \right\|_{\C^d} \\
&\leq \int_0^1 \left\| \exp(2i\pi(\eth tu + txs))\left(2i\pi (\eth u + xs) \xi(s + \eth ty) +  \eth y \xi'(s+\eth ty) \right) \right\|_{\C^d} \, dt \\
&=\int_0^1 \left\| 2i\pi (\eth u + xs) \xi(s + \eth ty) +  \eth y \xi'(s+\eth ty) \right\|_{\C^d} \, dt \\
&\leq \int_0^1 \|(u,x,y)\|_1\max\left\{\begin{array}{l}
\left\|2i\pi \eth \xi(s + \eth ty)\right\|_{\C^d} \text{,}\\
\left\|2i\pi s \xi(s + \eth ty)\right\|_{\C^d} \text{,}\\
\left\|\eth\xi'(s + \eth ty)\right\|_{\C^d}
\end{array} \right\} \,dt\\
&\leq 2\pi \max\{1,\eth_+\}\|(x,y,u)\|_1 \int_0^1 \mathrm{b}(s,ty) \,dt \\
&\leq 2\pi \max\{1,\eth_+\} \|(x,y,u)\|_1 \left(\sup_{y \in [-C,C]}\mathrm{b}(s,y)\right)\text{.}
\end{split}
\end{multline*}

Since:
\begin{equation}\label{strongly-continuous-pf-eq1}
\lim_{s\rightarrow\pm\infty} (1 + s^2)\sup_{y \in [-C,C]}\mathrm{b}(s,y) = M\text{,}
\end{equation}
we conclude that there exists $R > 0$ such that for all $s\in\R\setminus[-R,R]$, we have:
\begin{equation*}
\|\alpha_{\eth,d}^{x,y,u}\xi(s) - \xi(s)\|_{\C^d} \leq \frac{M_1 \|(x,y,u)\|_1}{1 + s^2}
\end{equation*}
for $M_1 = 4 M \pi \max\{1,\eth_+\}$.
We note that $M_1$ depends only on $M$, $\eth_+$ and $C$ through Expression (\ref{strongly-continuous-pf-eq1}), and not on $\xi$.

Since $s\in\R \mapsto \frac{1}{1+s^2}$ is continuous and strictly positive, we may adjust $M_1$ to a larger value if necessary such that:
\begin{equation*}
\min_{s\in [-R,R]} \frac{M_1}{1 + s^2} \geq 2\pi M\max\{1,\eth_+\} \text{.}
\end{equation*}

Therefore, we have, for all $s\in \R$ and $(x,y,u) \in \R^3$ with $\|(x,y,u)\|_1 \leq C$:
\begin{equation*}
\|\alpha_{\eth,d}^{x,y,u}\xi(s) - \xi(s)\|_{\C^d} \leq \frac{M_1 \|(x,y,u)\|_1}{1 + s^2} \leq \frac{M_1  C}{1 + s^2} \text{.}
\end{equation*}

Now, all the above computations may be applied equally well to $\xi'$ and $\xi''$. We conclude that indeed, Expression (\ref{strongly-continuous-eq1}) holds as stated.

Let now $\xi \in \mathcal{S}\otimes\C^d$ be chosen. Since $\xi$ is a Schwarz function, there exists $M > 0$ such that for all $s\in \R$, we have:
\begin{multline*}
\max\left\{ \|\xi(s)\|_{\C^d}, \|s\xi(s)\|_{\C^d}, \|\xi'(s)\|_{\C^d}, \|s\xi'(s)\|_{\C^d}, \|\xi''(s)\|_{\C^d}, \|s\xi''(s)\|_{\C^d} \right\} \\ \leq \frac{M}{1 + s^2}\text{.}
\end{multline*}
Thus we can apply our previous work to conclude that Expression (\ref{strongly-continuous-eq1}) holds for some $K > 0$, having chosen $C = 1$ for this last part of our proof.

Furthermore, we can apply now Lemma (\ref{l1-cont-lemma}). For this part, we pick $\eth > 0$; we need  not to worry about the uniformity in $\eth$ (we may as well assume $\eth_- = \eth_+ = \eth$ here). Thus, if $(x_n,y_n,u_n)_{n\in\N}$ converges to $0$, Lemma (\ref{l1-cont-lemma}) implies that:
\begin{equation*}
\begin{cases}
0 \leq \|\alpha_{\eth,d}^{x_n,y_n,u_n}\xi - \xi\|_{\HeisenbergMod{\theta}{p,q,d}} \leq \sqrt{\|\inner{\alpha_{\eth,d}^{x_n,y_n,u_n}\xi - \xi}{\alpha_{\eth,d}^{x_n,y_n,u_n}\xi - \xi}{\HeisenbergMod{\theta}{p,q,d}}\|_{\ell_1(\Z^2)}} \\
\lim_{n\rightarrow\infty} \sqrt{\|\inner{\alpha_{\eth,d}^{x_n,y_n,u_n}\xi - \xi}{\alpha_{\eth,d}^{x_n,y_n,u_n}\xi - \xi}{\HeisenbergMod{\theta}{p,q,d}}\|_{\ell_1(\Z^2)}} = 0
\end{cases}
\end{equation*}
which concludes the proof of our proposition for $\eth > 0$.

To prove our result for a general $\eth \not= 0$, we simply observe that for all $(x,y,u) \in \R^3$ we have $\alpha_{\eth,d}^{x,y,u} = \alpha_{-\eth,d}^{x,-y,-u}$ and thus our proposition is completely proven.
\end{proof}

We wish to use the actions of $\HeisenbergGroup$ on Heisenberg modules to define our D-norms. The next section presents a general source of possible D-norms from actions of Lie groups satisfying the properties we have established in this section.

\section{Seminorms from Lie group actions}

Connes introduced a quantized differential calculus on quantum tori in \cite{Connes80} using the dual action of the tori, using the Lie group structure of the tori. Moreover, he introduced a noncommutative connection on Heisenberg modules, and these connections proved to be solutions of the Yang-Mills problem for quantum $2$-tori \cite{ConnesRieffel87}. These connections were also useful in Rieffel's work on the classification of modules over quantum tori \cite{Rieffel83}.

Moreover, ergodic actions of metric compact groups on C*-algebras were the first example of L-seminorms constructed by Rieffel in \cite{Rieffel98a}. In this section, we begin investigating how to build D-norms from Lie group actions. We will employ as assumptions the properties which we derived for the action of the Heisenberg group on Heisenberg modules. Our construction, as we shall see, lies at the intersection of the purely metric picture of Rieffel and the differential picture of Connes, and is a noncommutative version of \cite[Example 3.10]{Latremoliere16c}.

Our D-norm will be constructed using the following theorem.

\begin{definition}\label{alpha-differentiable-def}
Let $\alpha$ be a strongly continuous action of a Lie group $G$ on a Banach space $\module{E}$. Let $\mathfrak{w}$ be a nonzero subspace of the Lie algebra of $G$. An element $\xi\in\module{E}$ is \emph{$\alpha$-differentiable with respect to $\mathfrak{w}$} when for all $X\in\mathfrak{w}$, the limit:
\begin{equation*}
X(\xi) = \lim_{t\rightarrow 0} \frac{\alpha^{\exp(t X)}\xi - \xi}{t}
\end{equation*}
exists.
\end{definition}

In any vector space $E$, and for any function $f: E\rightarrow\R$, we denote as usual:
\begin{equation*}
\limsup_{x \rightarrow 0} f(x) = \inf_{\delta>0} \sup\left\{ f(x) : 0 < \|x\| \leq \delta \right\} \text{.}
\end{equation*}

\begin{theorem}\label{seminorm-connection-thm}
Let $\alpha$ be a strongly continuous action by linear isometries of a Lie group $G$ on a Banach space $\module{E}$. Let $\mathfrak{g}$ be the Lie algebra of $G$ and let $\mathfrak{h} \subseteq \mathfrak{g}$ be a nonzero subspace of $\mathfrak{g}$.

Let $\module{S}\subseteq \module{E}$ be the subspace of $\module{E}$ consisting of $\alpha$-differentiable elements of $\module{E}$ with respect to $\mathfrak{h}$. We note that $\module{S}$ is dense in $\module{E}$.

Let $\|\cdot\|$ be a norm on $\mathfrak{h}$. For all $\xi \in \module{S}$, the norm of the linear map:
\begin{equation*}
\nabla\xi : X \in \mathfrak{h} \mapsto \nabla_X \xi = X(\xi) 
\end{equation*}
is denoted by $\opnorm{\nabla \xi}{}{}$.

If $\xi \in \module{S}$, then, for any $\delta > 0$:
\begin{equation*}
\begin{split}
\opnorm{\nabla \xi}{}{} &=\sup\left\{\frac{\left\|\alpha^{\exp(X)}\xi - \xi\right\|_{\module{E}}}{\|X\|} : X \in \mathfrak{h}\setminus\{0\} \right\} \\ 
&=\sup\left\{\frac{\left\|\alpha^{\exp(X)}\xi - \xi\right\|_{\module{E}}}{\|X\|} : X \in \mathfrak{h}\setminus\{0\}, \|X\| \leq \delta \right\}\\
&=\limsup_{X \rightarrow 0} \frac{\left\|\alpha^{\exp(X)}\xi - \xi\right\|_{\module{E}}}{\|X\|}
\text{.}
\end{split}
\end{equation*}

\end{theorem}

\begin{proof}
A smoothing argument \cite{Bratteli79} proves that the set:
\begin{equation*}
\left\{ \xi \in \module{E} : t > 0 \mapsto \frac{\alpha^{\exp(tX)}\xi - \xi}{t} \text{ has a limit at $0$ for all $X\in \mathfrak{g}$} \right\}
\end{equation*}
is dense in $\module{E}$. Therefore, since $\module{S}$ contains this set, $\module{S}$ is dense in $\module{E}$ as well.

Fix $\xi \in \module{S}$. Let $X \in \mathfrak{h}$. We define:
\begin{equation*}
F : t \in \R \mapsto \alpha^{\exp(tX)}\xi \text{.}
\end{equation*}

The function $F$ is continuously differentiable, and in particular, $F(0) = \xi$ and $F(1) = \alpha^{\exp(X)}\xi$. 

Moreover, using the fact that $t\in\R\mapsto \exp(tX)$ is a continuous group homomorphism:
\begin{equation*}
F'(t) = \lim_{s\rightarrow 0} \frac{\alpha^{\exp((t + s)X)} \xi - \alpha^{\exp(tX)}\xi}{h} = \lim_{s\rightarrow 0} \frac{\alpha^{\exp(tX)}\left(\alpha^{\exp(h X)} \xi - \xi\right)}{h} = \alpha^{\exp(tX)}\nabla_X \xi \text{.}
\end{equation*}

Thus:
\begin{equation*}
\alpha^{\exp(X)}\xi - \xi = \int_0^1 F'(t) \, dt = \int_0^1 \alpha^{\exp(tX)}\left(\nabla_X \xi \right) \,dt
\end{equation*}
so that:
\begin{equation*}
\begin{split}
\frac{\left\|\alpha^{\exp(X)}\xi - \xi\right\|_{\module{E}}}{\|X\|} &= \frac{\left\|\int_0^1 F'(t) \, dt\right\|_{\module{E}}}{\|X\|} \\
&\leq \frac{1}{\|X\|}\int_0^1 \left\|\alpha^{\exp(X)}\left(\nabla_X\xi\right)\right\|_{\module{E}} \,dt\\
&= \frac{1}{\|X\|} \int_0^1 \left\|\nabla_X\xi\right\|_{\module{E}} \,dt\text{ since $\alpha^{\exp(tX)}$ is an isometry by hypothesis,}\\
&\leq \frac{1}{\|X\|} \int_0^1 \opnorm{\nabla \xi}{}{} \|X\| \, dt = \opnorm{\nabla\xi}{}{} \text{.}
\end{split}
\end{equation*}
This proves that:
\begin{equation*}
\sup\left\{\frac{\left\|\alpha^{\exp(X)}\xi - \xi\right\|_{\module{E}}}{\|X\|} : X \in \mathfrak{h}\setminus\{0\} \right\} \leq \opnorm{\nabla\xi}{}{} \text{.}
\end{equation*}

On the other hand, let us now fix some $\delta > 0$. let us now assume that $\|X\| = 1$. We first note that:
\begin{equation*}
\begin{split}
\nabla_X\xi &= F'(0)\\
&= \lim_{t\rightarrow 0} \frac{F(t) - F(0)}{t} \text{ where $\lim$ is used for the topology of $(\module{E},\|\cdot\|_{\module{E}})$,} \\
&= \lim_{t\rightarrow 0} \frac{\alpha^{\exp(tX)}\xi - \xi}{t\|X\|} = \lim_{t\rightarrow 0} \frac{\alpha^{\exp(tX)}\xi - \xi}{\|tX\|} \text{.}
\end{split}
\end{equation*}
Thus for all $X\in\mathfrak{h}$ with $\|X\|=1$, since $\|tX\| \leq \delta$ for all $t\in\R$ with $|t| < \delta$:
\begin{equation*}
\left\|\nabla_X\xi\right\| \leq \sup\left\{\frac{\left\|\alpha^{\exp(Y)}\xi - \xi\right\|_{\module{E}}}{\|Y\|} : Y \in \mathfrak{h}\setminus\{0\}, \|Y\|\leq \delta \right\}
\end{equation*}
and thus:
\begin{equation*}
\begin{split}
\opnorm{\nabla \xi}{}{} &\leq \sup\left\{\frac{\left\|\alpha^{\exp(X)}\xi - \xi\right\|_{\module{E}}}{\|X\|} : X \in \mathfrak{h}\setminus\{0\},\|X\|\leq\delta \right\} \\
&\leq \sup\left\{\frac{\left\|\alpha^{\exp(X)}\xi - \xi\right\|_{\module{E}}}{\|X\|} : X \in \mathfrak{h}\setminus\{0\} \right\}\text{.}
\end{split}
\end{equation*}
We have thus concluded our argument, as the function:
\begin{equation*}
\delta \in (0,\infty) \mapsto \sup\left\{\frac{\left\|\alpha^{\exp(X)}\xi - \xi\right\|_{\module{E}}}{\|X\|} : X \in \mathfrak{h}\setminus\{0\},\|X\|\leq\delta \right\}
\end{equation*}
has been shown to be constant.
\end{proof}

We note that the seminorms constructed in Theorem (\ref{seminorm-connection-thm}) include Rieffel's L-seminorms in \cite{Rieffel98a} from actions of compact Lie groups.

\begin{corollary}\label{compact-seminorm-connection-cor}
Let $\alpha$ be a strongly continuous action by linear isometries of a compact connected Lie group $G$ on a Banach space $\module{E}$. As a compact Lie group, $G$ admits an $\mathrm{Ad}$-invariant inner product $\inner{\cdot}{\cdot}{\mathfrak{g}}$ on $\mathfrak{g}$. Let $\|\cdot\|$ be the norm associated with $\inner{\cdot}{\cdot}{\mathfrak{g}}$. For any $g \in G$, since $G$ is connected and compact, we may define $\ell(g)$ as the distance from $1_G$ to $g$ for the Riemannian metric induced by $\inner{\cdot}{\cdot}{\mathfrak{g}}$. 

If $\xi \in \module{S}$ then:
\begin{equation*}
\sup\left\{ \frac{\left\|\alpha^g\xi - \xi\right\|_{\module{E}}}{\ell(g)} : g \in G\setminus\{1_G\} \right\} = \opnorm{\nabla \xi}{}{} \text{.}
\end{equation*}
\end{corollary}

\begin{proof}
As $G$ is a compact group, it admits a right Haar probability measure $\mu$. Let $\inner{\cdot}{\cdot}{}$ be any inner product on $\mathfrak{g}$. If we set, for all $X,Y \in \mathfrak{g}$:
\begin{equation*}
\inner{X}{Y}{G} = \int_G \inner{\mathrm{Ad}_g X}{\mathrm{Ad}_g Y}{} \,d\mu(g)
\end{equation*}
then one easily verifies that $\inner{\cdot}{\cdot}{G}$ is an $\mathrm{Ad}$-invariant inner product on $\mathfrak{g}$.

Now, we endow $G$ with the Riemannian metric induced by left translation of the inner product $\inner{\cdot}{\cdot}{G}$. As this metric is induced by an $\mathrm{Ad}$-invariant inner product, it is in fact right invariant as well.

In particular, $G$, as a connected compact Riemannian manifold, is geodesically complete by Hopf-Rinow theorem. As a first application, we let $\ell(g)$ be the distance from $1_G$ to $g$ in $G$ for this Riemannian metric, for all $g\in G$. As a second application, we note that the Riemannian exponential map of $G$ for our metric is indeed surjective.

It is now possible to check that the exponential map for the Lie group $G$ and the exponential map for the Riemannian metric coincide. This is done by checking that the Riemannian exponential map defines a $1$-parameter subgroup of $G$.

With this in mind, we conclude that for all $X\in \mathfrak{g}$, we have:
\begin{equation*}
\ell(\exp(X)) = \inf\left\{ \|Y\| : \exp(X) = \exp(Y) \right\} \text{.}
\end{equation*}
We note that the Lie exponential map is certainly not injective, at least as long as $G$ is of dimension at least one, though this does not affect our conclusion.

Moreover, since $G$ is a compact connected Lie group, $\exp$ is surjective since the Riemannian exponential is surjective. Thus, our corollary is proven using Theorem (\ref{seminorm-connection-thm}). 
\end{proof}

Now, Rieffel proved in \cite{Rieffel98a} that the obvious necessary condition for a seminorm of the type given in Corollary (\ref{compact-seminorm-connection-cor}) to be a L-seminorm is, remarkably, sufficient as well. This fact is highly non-trivial as well, and we record it here as it will be the source of quantum metrics we put on quantum tori.

\begin{theorem}[{\cite[Theorem 1.9]{Rieffel98a}}]\label{Rieffel-L-seminorm-thm}
Let $\beta$ be a strongly continuous group action by *-automorphisms of a compact group $G$ on a unital C*-algebra $\A$. Let $\ell$ be a continuous length function on $G$. For all $a\in\A$, we define:
\begin{equation*}
\Lip(a) = \sup\left\{ \frac{\|\beta^g(a) - a\|_\A}{\ell(g)} : g\in G\setminus\{e\} \right\}\text{,}
\end{equation*}
allowing for this quantity to be infinite. Then the following are equivalent:
\begin{enumerate}
\item $(\A,\Lip)$ is a quantum compact metric space (which is necessarily Leibniz),
\item $\{a\in\A : \forall g \in G \quad \beta^g(a) = a\} = \C\unit_\A$.
\end{enumerate}
\end{theorem}

We note that the proof of Theorem (\ref{Rieffel-L-seminorm-thm}) involves explicitly the fact that the spectral subspaces of the action $\beta$ are finite dimensional under the condition of ergodicity \cite{Hoegh-Krohn81}. This result is not trivial, and worse yet for our purpose, does not carry to locally compact group. In fact, besides the trivial representation, no irreducible representation of the Heisenberg group is finite dimensional --- so we are as far as we can to apply the idea in \cite{Rieffel98a}. In this paper, we shall focus on the Heisenberg modules, and we will prove in this case that the seminorms constructed in Theorem (\ref{seminorm-connection-thm}) have compact unit balls using quite different techniques from Rieffel.

The rest of this section introduces the general scheme to construct D-norms from Lie group actions which we will employ in this paper, and prove that this construction meets all our requirements except, maybe, for the compactness of the unit ball which, in the case of Heisenberg modules, will be the subject  of our next section.

\begin{proposition}\label{Leibniz-prop}
Let $\beta$ be the action of a compact connected Lie group $G$ on a unital C*-algebra $\A$ via *-automorphisms. Let $\alpha$ be the action by isometric $\C$-linear isomorphisms of a Lie group $H$ on a Hilbert module $(\module{M},\inner{\cdot}{\cdot}{\module{M}})$ over $\A$. We write $\mathfrak{g}$ and $\mathfrak{h}$ the respective Lie algebras of $G$ and $H$, and $\exp_G : \mathfrak{g}\rightarrow G$ and $\exp_H : \mathfrak{h}\rightarrow H$ be the respective Lie exponential maps of $G$ and $H$.

Let $\mathfrak{w}$ be a nonzero subspace of $\mathfrak{h}$. Let $\|\cdot\|_\flat$ be a norm on $\mathfrak{g}$ and $\|\cdot\|^\sharp$ be a norm on $\mathfrak{w} \subseteq \mathfrak{h}$.

We set for all $a\in \A$:
\begin{equation*}
\Lip(a) = \sup\left\{\frac{\left\|\beta^{\exp(X)}a - a\right\|_\A}{\|X\|_\flat} : X \in \mathfrak{g} \setminus \{0\} \right\} \text{,}
\end{equation*}
and for all $\xi \in \module{E}$:
\begin{equation*}
\CDN(\xi) = \sup\left\{\|\xi\|_{\module{M}}, \frac{\left\|\alpha^{\exp(X)}\xi - \xi\right\|_{\module{M}}}{\|X\|^\sharp} : X \in \mathfrak{w} \setminus \{0\} \right\} \text{.}
\end{equation*}

If there exist two linear maps $j : \mathfrak{w} \rightarrow \mathfrak{g}$ and $q : \mathfrak{g} \rightarrow \mathfrak{w}$ such that:
\begin{enumerate}
\item for all $\xi,\omega\in\module{M}$ and $X\in\mathfrak{w}$:
\begin{equation}\label{compatible-actions-cond1}
\beta^{\exp_G(X)}\inner{\xi}{\omega}{\module{M}} = \inner{\alpha^{\exp_H(j(X))} \xi}{\alpha^{\exp_H(j(X))}\omega}{\module{E}}
\end{equation}
and:
\begin{equation}\label{compatible-actions-cond2}
\alpha^{\exp_H(X)}(a\xi) = \beta^{\exp_G(q(X))}(a)\alpha^{\exp_H(X)}\xi \text{,}
\end{equation}
\item $j$ is an isometry from $(\mathfrak{g},\|\cdot\|_\flat)$ to $(\mathfrak{w},\|\cdot\|^\sharp)$,
\item $q$ is a surjection of norm at most $1$, i.e. $\|q(X)\|_\flat \leq \|X\|^\sharp$ for all $X\in\mathfrak{w}$,
\end{enumerate}
then:
\begin{enumerate}
\item $\Lip$ is a seminorm on a dense subspace of $(\A,\|\cdot\|_\A)$, and moreover:
\begin{equation*}
\Lip(a) = 0 \iff \forall g\in G\quad \beta^g(a) = a \text{,}
\end{equation*}
\item $\CDN$ is a norm on a dense subspace of $(\module{M},\inner{\cdot}{\cdot}{\module{M}})$ and $\CDN(\cdot)\geq\|\cdot\|_{\module{M}}$,
\item $\Lip$ and $\CDN$ are lower semicontinuous,
\item for all $a\in\A$ and $\xi \in \module{M}$:
\begin{equation*}
\CDN(a\xi) \leq \|a\|_\A \CDN(\xi) + \Lip(a) \|\xi\|_{\module{M}}\text{,}
\end{equation*}
\item for all $\xi,\omega\in\module{M}$:
\begin{equation*}
\Lip\left(\inner{\xi}{\omega}{\module{M}}\right) \leq \|\xi\|_{\module{M}}\CDN(\omega) + \CDN(\xi)\|\omega\|_{\module{M}} \text{.}
\end{equation*}
\end{enumerate}
\end{proposition}

\begin{proof}
Let $\module{S}_{\mathfrak{g}}(\A)$ be the subspace of $\A$ consisting of all the $\beta$-differentiable elements with respect to $\mathfrak{g}$, and $\module{S}_{\mathfrak{h}}(\module{M})$ be the subspace of $\module{M}$ consisting of all the $\alpha$-differentiable elements of $\module{M}$ with respect to $\mathfrak{w}$.

For any $a\in \module{S}_{\mathfrak{g}}(\A)$, we define the linear map $\partial a : X \in \mathfrak{g} \mapsto X(a)$ whose norm is denoted by $\opnorm{\partial a}{\mathfrak{g}}{\A}$, where $\mathfrak{g}$ is endowed with $\|\cdot\|_\flat$. Since $\mathfrak{g}$ is finite dimensional, $\partial a$ is continuous and thus has finite norm for all $a\in\module{S}_{\mathfrak{g}}(\A)$.

For any $\xi \in \module{S}_{\mathfrak{w}}(\module{M})$, we also define $\nabla \xi : X \in \mathfrak{w} \mapsto X(\xi)$ whose norm is $\opnorm{\nabla \xi}{\mathfrak{w}}{\module{M}}$ where $\mathfrak{w}$ is endowed by $\|\cdot\|^\sharp$ --- since $\mathfrak{w}$ is finite dimensional, the norm of $\nabla \xi$ is finite as well.

By Theorem (\ref{seminorm-connection-thm}), for all $a\in \module{S}_{\mathfrak{g}}(\A)$ and for all $\xi \in \module{S}_{\mathfrak{w}}(\module{M})$, then:
\begin{equation*}
\Lip(a) = \opnorm{\partial a}{\mathfrak{g}}{\A} < \infty \text{ and }\CDN(\xi) = \opnorm{\nabla \xi}{\mathfrak{w}}{\module{M}} < \infty\text{.}
\end{equation*}
Since $\module{S}_{\mathfrak{g}}(\A)$ and $\module{S}_{\mathfrak{w}}(\module{E})$ are dense, we conclude that the domains of $\Lip$ and $\CDN$ are indeed dense.

\bigskip

Since $\CDN(\cdot) \geq \|\cdot\|_{\module{M}}$ by construction, $\CDN$ is in particular a norm on its domain.

\bigskip

Moreover if $\Lip(a) = 0$ for some $a\in\A$, we immediately conclude that $\beta^{g}a = a$ for all $g \in G$ since the exponential map of $G$ is surjective.

\bigskip

The function $\xi \in \module{M} \mapsto \frac{\alpha^{\exp(X)}\xi - \xi}{\|X\|^\sharp}$ is continuous for all $X \in \mathfrak{w} \setminus\{0\}$ and thus $\CDN$ is lower semi-continuous as the pointwise supremum of continuous functions. The same reasoning and conclusion applies to $\Lip$.

\bigskip

We are left to prove the two forms of the Leibniz inequalities, which can be easily checked by direct computation. Let $\xi,\omega \in \module{M}$. We compute:
\begin{align*}
\Lip\left(\inner{\xi}{\omega}{\module{E}}\right) &= \sup\left\{ \frac{\left\| \beta^{\exp(X)}\inner{\xi}{\omega}{\module{E}} - \inner{\xi}{\omega}{\module{E}}\right\|_\A}{\|X\|_\flat} : X \in \mathfrak{g}\setminus\{0\} \right\} \\
&=\sup\left\{ \frac{\left\| \inner{\alpha^{\exp(j(X))}\xi}{\alpha^{\exp(j(X))}\omega}{\module{E}} - \inner{\xi}{\omega}{\module{E}}\right\|_\A}{\|j(X)\|^\sharp} : X \in \mathfrak{g}\setminus\{0\} \right\} \\
&\leq \sup\left\{ \frac{\left\| \inner{\alpha^{\exp(X)}\xi}{\alpha^{\exp(X)}\omega}{\module{E}} - \inner{\xi}{\omega}{\module{E}}\right\|_\A}{\|X\|^\sharp} : X \in \mathfrak{w}\setminus\{0\} \right\} \\
&\leq \sup\left\{ \frac{\left\| \inner{\alpha^{\exp(X)}\xi}{\alpha^{\exp(X)}\omega}{\module{E}} - \inner{\alpha^{\exp(X)}\xi}{\omega}{\module{E}}\right\|_\A}{\|X\|^\sharp} : X \in \mathfrak{w}\setminus\{0\} \right\} \\
&\quad + \sup\left\{ \frac{\left\| \inner{\alpha^{\exp(X)}\xi}{\omega}{\module{E}} - \inner{\xi}{\omega}{\module{E}}\right\|_\A}{\|X\|^\sharp} : X \in \mathfrak{w}\setminus\{0\} \right\} \\
&\leq \sup\left\{ \frac{\left\|\alpha^{\exp(X)}\xi\right\|_{\module{M}}\left\|\alpha^{\exp(X)}\omega-\omega\right\|_{\module{E}}}{\|X\|^\sharp} : X \in \mathfrak{w}\setminus\{0\} \right\} \\
&\quad + \sup\left\{ \frac{\left\|\alpha^{\exp(X)}\xi - \xi\right\|_{\module{E}}}{\|X\|^\sharp} : X \in \mathfrak{w}\setminus\{0\} \right\}\left\|\omega\right\|_{\module{M}}\\
&\leq \left\|\xi\right\|_{\module{M}} \sup\left\{ \frac{\left\|\alpha^{\exp(X)}\omega-\omega\right\|_{\module{E}}}{\|X\|^\sharp} : X \in \mathfrak{w}\setminus\{0\} \right\} \\
&\quad + \sup\left\{ \frac{\left\|\alpha^{\exp(X)}\xi - \xi\right\|_{\module{E}}}{\|X\|^\sharp} : X \in \mathfrak{w}\setminus\{0\} \right\}\left\|\omega\right\|_{\module{M}}\\
&= \left\|\xi\right\|_{\module{M}}\CDN(\omega) + \CDN(\xi)\left\|\omega\right\|_{\module{M}} \text{.} 
\end{align*}

Now, let $a\in\A$ and $\xi \in \module{M}$. We compute:
\begin{align*}
\sup&\,\left\{\frac{\left\|\alpha^{\exp(X)}\left(a\xi\right) - a\xi \right\|_{\module{M}}}{\|X\|^\sharp} : X \in \mathfrak{w}\setminus\{0\} \right\} \\
&= \sup\left\{\frac{\left\|\beta^{\exp(q(X))}(a) \alpha^{\exp(X)}\left(\xi\right) - a\xi \right\|_{\module{M}}}{\|X\|^\sharp} : X \in \mathfrak{w}\setminus\{0\} \right\} \\
&\leq \sup\left\{\frac{\left\|\beta^{\exp(q(X))}(a) \alpha^{\exp(X)}\left(\xi\right) - a\alpha^{\exp(X)}\xi \right\|_{\module{M}}}{\|q(X)\|_\flat} : X \in \mathfrak{w}\setminus\ker q \right\} \\
& \quad + \sup\left\{\frac{\left\|a \alpha^{\exp(X)}\left(\xi\right) - a\xi \right\|_{\module{M}}}{\|X\|^\sharp} : X \in \mathfrak{w}\setminus\{0\} \right\} \\
&\leq \sup\left\{\frac{\left\|\beta^{\exp(q(X))}(a) - a \right\|_{\module{M}}}{\|X\|_\flat} : X \in \mathfrak{g}\setminus\{0\} \right\} \|\xi\|_{\module{M}} + \|a\|_\A \CDN(\xi) \\
&= \Lip(a)\|\xi\|_{\module{M}} + \|a\|_\A\CDN(\xi)\text{,}
\end{align*}
as desired.
\end{proof}

Thus, Proposition (\ref{Leibniz-prop}) shows that if we follow the scheme suggested by Theorem (\ref{seminorm-connection-thm}), then we obtain potential D-norms on modules. The missing property is the compactness of the closed unit ball for the D-norm candidate.

We conclude our section by connecting our metric framework with the noncommutative differential framework of connections on modules. Let us use the notations of Proposition (\ref{Leibniz-prop}). A direct computation shows that for all $X \in \mathfrak{w}$, the following holds:
\begin{equation}\label{connection-modular-Leibniz-eq}
\nabla_X (a\xi) = q(X)a \cdot \xi + a \nabla_X \xi
\end{equation}
while for all $X \in \mathfrak{g}$, we also have:
\begin{equation}\label{connection-inner-Leibniz-eq}
X(\inner{\xi}{\omega}{\module{M}}) = \inner{j(X)\xi}{\omega}{\module{M}} + \inner{\xi}{j(X)\omega}{\module{M}}\text{.}
\end{equation}

We also denote $\A\otimes\mathfrak{g}^\ast$ by $\Omega_1$ and the space of $\beta$-differentiable elements of $\A$ by $\A_1$. We define $\partial : \A_1 \rightarrow \Omega_1$ by setting, for all $a\in\A_1$:
\begin{equation*}
\partial a : X \in \mathfrak{g} \mapsto X(a) \text{.}
\end{equation*}
We observe trivially that $\Omega_1$ is an $\A$-$\A$-bimodule and that $\partial$ is a derivation, i.e. $\partial(ab) = a\partial(b) + \partial(a)b$ for all $a,b \in \A_1$.

We first note that to get an interesting connection, we want $q$ to be injective, i.e. $\mathfrak{g}$ and $\mathfrak{w}$ to be isomorphic. It is always possible to increase the dimension of $\mathfrak{g}$ (the Lie algebra structure is actually not involved in the computations to follow, so this is always possible), but this would amount to define $\partial_X = 0$ for all vector $X$ not in $\mathfrak{g}$, and this is rather awkward and artificial.

Since, for the differential picture, the norms $\|\cdot\|_\flat$ and $\|\cdot\|^\sharp$ do not play a role in the construction of the connection, we will for now identify $\mathfrak{g}$ and $\mathfrak{w}$ and $j$ and $q$ with the identity map.

With this assumption, Expressions (\ref{connection-modular-Leibniz-eq}) translates to the operator $\nabla : \module{M} \rightarrow \module{M}\otimes\mathfrak{g}^\ast$, defined by:
\begin{equation*}
\nabla (\xi) : X \in \mathfrak{g} \mapsto \nabla_X \xi
\end{equation*}
for all $\alpha$-differentiable $\xi \in \module{M}$ with respect to $\mathfrak{g}$, to be a noncommutative connection. We indeed easily check that for all $a\in\A$ and $\xi\in\module{M}$:
\begin{equation*}
\nabla(a\xi) = a\nabla(\xi) + \partial(a)\xi\text{.}
\end{equation*}

Expression (\ref{connection-inner-Leibniz-eq}) means that the connection $\nabla$ is hermitian, i.e. it is compatible with the noncommutative equivalent of a metric on the quantum vector bundle $\module{M}$. It is tempting to call $\nabla$ a Levi-Civita connection, although we do not address here the computation of the torsion of $\nabla$. Nonetheless, we see that our structure provides a noncommutative Riemannian geometry. This is the structure which inspired our definition of {\gQVB}, and we now can see how it is implemented through our main example.

In summary, we have constructed a natural D-norm candidate on modules carrying certain Lie group actions. The key difficulty, of course, regards the compactness of the unit ball of such a D-norm.

\section{A D-norm from a connection on Heisenberg modules}

We now define our D-norms on Heisenberg modules. Our method employs the idea of Theorem (\ref{seminorm-connection-thm}) and Proposition (\ref{Leibniz-prop}), where the actions of the Heisenberg group on Heisenberg modules defines a norm which restricts to the operator norm of a connection constructed via the associated action of the Heisenberg Lie algebra.

As noted at the end of the previous section, we want to only work with a subspace of the Heisenberg Lie algebra to build our D-norm and its associated connection, since the central element of the Heisenberg Lie algebra does not act, so to speak, as a derivation --- it simply acts by multiplication by a scalar. We follow a pattern which is common in the literature on the Heisenberg group: we only consider the action of the subspace $\mathrm{span}\{P,Q\}$ in the Lie algebra $\mathfrak{H}$. 

We thus endow $\mathrm{span}\{P,Q\}$ with a norm. If we were to construct a metric on the Heisenberg group using this data --- by defining the length of a curve whose tangent vector at (almost) every point lies in $\mathrm{span}\{P,Q\}$ in the usual manner by integrating the norm of the tangent vector along the curve, and then defining the distance between two points as the infimum of the length of all so-called horizontal curves --- we would actually obtain a sub-Finslerian metric (if our choice of norm comes from a Hilbert space structure, we would have a sub-Riemannian structure and our construction would give rise to a Carnot-Carath{\'e}dory distance on the Heisenberg group). 

However, as discussed, we do not transport the Carnot-Carath{\'e}dory metric from the Heisenberg group via its action in this paper. We prefer to carry the norm of the subspace $\mathrm{span}\{P,Q\}$ of the Heisenberg Lie algebra to our modules. This approach means that we work with a connection, and seems more natural. In essence, the Carnot-Caratheodory is the metric obtained on the group while our D-norms are the quantum metrics obtained on our modules; as the acting group is not compact, we have no reason to expect them to agree.

With this in mind, we now introduce:

\begin{definition}\label{Heisenberg-dnorm-def}
Let $p\in\Z$, $q \in \N\setminus\{0\}$ and $d\in q\N$ with $d > 0$. Let $\theta\in\R\setminus\left\{\frac{p}{q}\right\}$. Let $\|\cdot\|$ be a norm on $\R^2$. We endow the Heisenberg module $\HeisenbergMod{\theta}{p,q,d}$ with the norm:
\begin{equation*}
\CDN_{\theta}^{p,q,d}(\xi) = \sup\left\{ \|\xi\|_{\HeisenbergMod{\theta}{p,q,d}}, \frac{\left\|\alpha_{\eth,d}^{\exp_{\HeisenbergGroup}(x P  + y Q)}\xi - \xi\right\|_{\HeisenbergMod{\theta}{p,q,d}}}{2\pi|\eth| \|(x,y)\|} : (x,y) \in \R^2\setminus\{0\} \right\}
\end{equation*}
where $\eth = \theta - \frac{p}{q}$.
\end{definition}

We now lighten our notation for the rest of our paper.

\begin{convention}
We endow $\R^2$ with a fixed norm $\|\cdot\|$ for the rest of this paper. We shall denote $\CDN_\theta^{\|\cdot\|,p,q,d}$ simply by $\CDN_\theta^{p,q,d}$, as the norm on $\R^2$ will not be understood. We emphasize that $\|\cdot\|$ is \emph{independent} of any of the parameters $p,q,d$ and $\theta$.

The norm $\|\cdot\|$ on $\R^2$ provides us with a continuous length function on $\qt{\theta}$ for all $\theta\in\R$. This length function arises from the invariant Finslerian metric induced by $\|\cdot\|$. A direct computation simply shows that:
\begin{equation*}
\ell(\exp(ix),\exp(iy)) = \inf\{ \|(x + 2 n \pi, y + 2 m \pi)\| : n,m \in \Z^2 \} \text{.}
\end{equation*}

For all $\theta\in\R$, we denote by $\Lip_\theta$ the L-seminorm on $\qt{\theta}$ associated with the action $\beta_\theta$ on $\qt{\theta}$ and the length function $\ell$ via \cite[Theorem 1.9]{Rieffel98a}. We note that since $\T^2$ is compact and Abelian, Corollary (\ref{compact-seminorm-connection-cor}) implies that for all $a\in\qt{\theta}$:
\begin{equation*}
\Lip_\theta(a) = \sup\left\{\frac{\|\beta_\theta^{\exp_{\T^2}(x,y)}\xi - \xi\|_{\qt{\theta}}}{\|(x,y)\|} : (x,y) \in \R^2\setminus\{0\} \right\}
\end{equation*}
and $\Lip_\theta$ agrees with the operator norm of derivative for the natural differential calculus defined by $\beta_\theta$ on $\beta_\theta$-differentiable elements. We refer to the previous section for a discussion of these matters.
\end{convention}

We begin by listing various equivalent expressions for our D-norm candidates, as we shall use whichever may prove useful in this paper.

\begin{remark}
We recall from Notation (\ref{Schrodinger-notation}) that:
\begin{equation*}
\exp_{\HeisenbergGroup}\left(x P + y Q\right) = \left(x, y, \frac{1}{2}xy \right)
\end{equation*}
for all $x,y \in \R$.

For all $p,q\in\N$, $d\in q\N$ with $d > 0$, $\theta\in\R\setminus\{pq^{-1}\}$ and $\xi \in \HeisenbergMod{\theta}{p,q,d}$, the following identities hold:
\begin{equation*}
\begin{split}
\CDN_\theta^{p,q,d}(\xi) &= \sup\left\{ \|\xi\|_{\HeisenbergMod{\theta}{p,q,d}}, \frac{\left\|\alpha_{\eth,d}^{x,y,\frac{1}{2}xy}\xi - \xi\right\|_{\HeisenbergMod{\theta}{p,q,d}}}{2\pi|\eth|\|(x,y)\|} : (x,y) \in \R^2\setminus\{0\} \right\} \\
&= \sup\left\{ \|\xi\|_{\HeisenbergMod{\theta}{p,q,d}}, \frac{\left\| \sigma_{\eth,d}^{x,y}\xi - \xi\right\|_{\HeisenbergMod{\theta}{p,q,d}}}{2\pi|\eth| \|(x,y)\|} : (x,y) \in \R^2\setminus\{0\} \right\} \text{.}
\end{split}
\end{equation*}
\end{remark}

\begin{proposition}\label{DLeibniz-prop}
Let $p,q \in \N$ and $d\in q\N$ with $d > 0$. Let $\theta\in\R\setminus\left\{\frac{p}{q}\right\}$. 

We endow $\mathrm{span}\{P,Q\}$ with the norm $2\pi|\eth|\|\cdot\|$. We also define, for all $(x,y) \in \R^2$ and $\xi\in\module{S}_\theta^{p,q,d}$:
\begin{equation*}
\begin{split}
\nabla^\eth_{x,y} \xi &= \lim_{t\rightarrow 0} \frac{\alpha_{\eth,d}^{\exp_{\HeisenbergGroup}\left(t\left(x P + y Q\right)\right)}\xi - \xi}{t}  \\
&= \lim_{t\rightarrow 0}\frac{\alpha_{\eth,d}^{tx,ty,\frac{1}{2} t^2 xy}\xi - \xi}{t}\text{.}
\end{split}
\end{equation*}

To ease notation, let $\opnorm{\cdot}{}{2\pi|\eth|}$ denote the operator norm for linear maps from $(\R^2,2\pi|\eth|\|\cdot\|)$ to $(\HeisenbergMod{\theta}{p,q,d},\|\cdot\|_{\HeisenbergMod{\theta}{p,q,d}})$.

We record:
\begin{enumerate}
\item $\CDN_\theta^{p,q,d}$ is a norm on a dense subspace of $\HeisenbergMod{\theta}{p,q,d}$,
\item For all $\xi \in \module{S}_\theta^{p,q,d}$ and for all $\delta > 0$, the following expressions hold:
\begin{equation*}
\begin{split}
\CDN_\theta^{p,q,d}(\xi) &= \max\left\{\|\xi\|_{\HeisenbergMod{\theta}{p,q,d}}, \opnorm{\nabla^\eth \xi}{}{2\pi|\eth|} \right\} \\
&= \sup\left\{ \frac{\|\sigma_{\eth,d}^{x,y}\xi - \xi\|_{\HeisenbergMod{\theta}{p,q,d}}}{2\pi|\eth|\|(x,y)\|} : (x,y) \in \R^2, 0 < \|(x,y)\| < \delta \right\} \\
&= \limsup_{(x,y)\rightarrow 0} \frac{\left\|\sigma_{\eth,d}^{x,y}\xi - \xi\right\|_{\HeisenbergMod{\theta}{p,q,d}}}{2\pi|\eth|\|(x,y)\|}\text{.}
\end{split}
\end{equation*}
\item If $a\in\qt{\theta}$ and $\xi \in \HeisenbergMod{\theta}{p,q,d}$ then:
\begin{equation*}
\CDN_\theta^{p,q,d}(a\xi) \leq \|a\|_{\qt{\theta}} \CDN_\theta^{p,q,d}(\xi) + \Lip_\theta(a) \|\xi\|_{\HeisenbergMod{\theta}{p,q,d}} \text{.}
\end{equation*}
\item  If $\xi,\omega \in \HeisenbergMod{\theta}{p,q,d}$ then:
\begin{equation*}
\Lip_\theta\left(\inner{\xi}{\omega}{\HeisenbergMod{\theta}{p,q,d}}\right) \leq \|\xi\|_{\HeisenbergMod{\theta}{p,q,d}} \CDN_\theta^{p,q,d}(\omega) + \CDN_\theta^{p,q,d}(\xi) \|\omega\|_{\HeisenbergMod{\theta}{p,q,d}} \text{.}
\end{equation*}
\end{enumerate}
\end{proposition}

\begin{proof}
The Lie algebra of $\T^2$ is $\R^2$ with the exponential map given as:
\begin{equation*}
\exp_{\T^2} : (x,y) \in \R^2 \mapsto (\exp(ix),\exp(iy))\text{.}
\end{equation*}

Now, the map $\upsilon_\eth : (x,y) \in \R^2 \mapsto (2i\pi \eth y, -2i\pi \eth x)$ satisfies, according to Lemma (\ref{alpha-beta-lemma}), the relation:
\begin{equation*}
\beta_\theta^{\exp_{\T^2}(\upsilon_\eth(x,y))}\inner{\xi}{\omega}{\HeisenbergMod{\theta}{p,q,d}} = \inner{\sigma_{\eth,d}^{\exp_\HeisenbergGroup(x,y)}\xi}{\sigma_{\eth,d}^{\exp_\HeisenbergGroup(x,y,0)}\omega}{\HeisenbergMod{\theta}{p,q,d}}\text{.}
\end{equation*}
and, according to Lemma (\ref{alpha-morphism-lemma}), the relation:
\begin{equation*}
\sigma_{\eth,d}^{\exp_{\HeisenbergGroup}(x,y)}(a\xi) = \beta_\theta^{\exp_{\T^2}(\upsilon_\eth(x,y))}(a) \sigma_{\eth,d}^{\exp_{\HeisenbergGroup}(x,y)}(\xi)\text{.}
\end{equation*}

In order to apply Proposition (\ref{Leibniz-prop}), since $\upsilon_\eth$ is indeed a linear isomorphism, we endow $\mathrm{span}\{P,Q\}$ with the norm:
\begin{equation*}
\|xP + yQ\|_\ast = 2\pi|\eth|\|(x,y)\| \text{.}
\end{equation*}
We now are in the setting of Proposition (\ref{Leibniz-prop}), which allows us to conclude all but Assertion (2) in our proposition. Assertion (2), in turn, follows from Theorem (\ref{seminorm-connection-thm}), with our choice of norm.
\end{proof}

We now turn to the remaining, main issue of the compactness of the closed unit balls for our D-norm candidates. The strategy we employ relies on a particular source of finite rank operators naturally associated with the Sch{\"o}dinger representations of $\R^2$ via the Weyl calculus. 

Our first step is to introduce the convolution-like operators at the core of our analysis.

\begin{lemma}\label{operator-lemma}
Assume Hypothesis (\ref{hyp-1}). If $f \in L^1(\R^2)$ and:
\begin{equation*}
\sigma_{\eth,d}^f = \iint_{\R^2} f(x,y)\alpha_{\eth,d}^{x,y,\frac{x y}{2}} \, dx dy
\end{equation*}
then $\sigma_{\eth,d}^f$ is a well-defined operator on $\HeisenbergMod{\theta}{p,q,d}$ and $\opnorm{\sigma_{\eth,d}^f}{}{\HeisenbergMod{\theta}{p,q,d}} \leq \|f\|_{L^1(\R^2)}$.
\end{lemma}

\begin{proof}
Let $\xi \in \HeisenbergMod{\theta}{p,q,d}$. Using Lemma (\ref{alpha-isometry-lemma}), i.e. the fact that $\alpha_{\eth,d}^{x,y,u}$ is an isometry of $\HeisenbergMod{\theta}{p,q,d}$ for all $(x,y,u)\in\HeisenbergGroup$, we simply compute:
\begin{equation*}
\begin{split}
\iint_{\R^2} \left\| f(x,y) \alpha_{\eth,d}^{x,y,\frac{x y}{2}}(\xi)\right\|_{\HeisenbergMod{\theta}{p,q,d}} \,dx dy
&= \iint_{\R^2} |f(x,y)| \left\|\alpha_{\eth,d}^{x,y,\frac{x y}{2}}(\xi)\right\|_{\HeisenbergMod{\theta}{p,q,d}} \, dx dy\\
&= \iint_{\R^2} |f(x,y)| \left\|\xi\right\|_{\HeisenbergMod{\theta}{p,q,d}} \,dx dy\\
&= \|f\|_{L^1(\R^2)} \|\xi\|_{\HeisenbergMod{\theta}{p,q,d}}\text{.} 
\end{split}
\end{equation*}
Thus $\sigma_{\eth,d}^f$ is well-defined, and moreover:
\begin{equation*}
\left\|\sigma_{\eth,d}^f(\xi)\right\|_{\HeisenbergMod{\theta}{p,q,d}} = \left\| \iint_{\R^2} f(x,y) \sigma_{\eth,d}^{x,y}(\xi) \, dx dy\right\|_{\HeisenbergMod{\theta}{p,q,d}} \leq \|f\|_{L^1(\R^2)} \|\xi\|_{\HeisenbergMod{\theta}{p,q,d}} \text{.}
\end{equation*}
This completes our proof.
\end{proof}

We now prove the first of two core lemmas of this section, which provides us with a mean to approximate elements in Heisenberg modules using our convolution-type operators, in a manner which is uniform in our prospective D-norms. This lemma is an adjustment of \cite{Rieffel00} to our context.

\begin{lemma}\label{dnorm-approx-lemma}
Assume Hypothesis (\ref{hyp-1}). Let $\varepsilon > 0$. If $f : \R^2 \rightarrow [0,\infty)$ is measurable and satisfies:
\begin{enumerate}
\item $\int_{\R^2} f = 1$,
\item $\iint_{\R^2} f(x,y)\|(x,y)\| \, dx dy \leq \frac{\varepsilon}{2\pi|\eth|}$,
\end{enumerate}
then for all $\xi \in \HeisenbergMod{\theta}{p,q,d}$:
\begin{equation*}
\left\|\xi - \sigma_{\eth,d}^f \xi\right\|_{\HeisenbergMod{\theta}{p,q,d}} \leq \varepsilon \CDN_\theta^{p,q,d}(\xi) \text{.}
\end{equation*}
\end{lemma}

\begin{proof}
If $\xi \in \HeisenbergMod{\theta}{p,q,d}$, then:
\begin{equation*}
\begin{split}
\left\|\xi - \sigma_{\eth,d}^f \xi\right\|_{\HeisenbergMod{\theta}{p,q,d}} &= \left\|\iint_{\R^2} f(x,y) \xi\,dx dy  - \iint_{\R^2} f(x,y) \alpha_{\eth,d}^{x,y,\frac{x y}{2}} \xi\, dx dy \right\|_{\HeisenbergMod{\theta}{p,q,d}} \\
&\leq \iint_{\R^2} f(x,y) \|\xi - \alpha_{\eth,d}^{x,y,\frac{x y}{2}} \xi\|_{\HeisenbergMod{\theta}{p,q,d}} \, dx dy \\
&\leq \iint_{\R^2} f(x,y) 2\pi|\eth| \|(x,y)\| \frac{\|\xi - \alpha_{\eth,d}^{x,y,\frac{x y}{2}} \xi\|_{\HeisenbergMod{\theta}{p,q,d}}}{2\pi|\eth|\|(x,y)\|} \, dx dy \\
&\leq \iint_{\R^2}  f(x,y) 2\pi|\eth|\|(x,y)\| \CDN_\theta^{\rho} (\xi) \, dx dy  \\
&= \CDN_\theta^{p,q,d}(\xi) \left( 2\pi|\eth| \frac{\varepsilon}{2\pi|\eth|} \right) \\ 
&= \varepsilon \CDN_{\theta}^\rho(\xi) \text{,}
\end{split}
\end{equation*}
as desired.
\end{proof}

We now ensure that we indeed have an ample source of functions which meet the hypothesis of Lemma (\ref{dnorm-approx-lemma}).

\begin{notation}
If $(E,d)$ is a metric space then the closed ball $\{x\in E : d(x_0,x) \leq r\}$ of center $x_0 \in E$ and radius $r \geq 0$ is denoted by $E[x_0,r]$.
\end{notation}

The following lemma is valid for any norm on $\R^2$; we shall work within our context with the fixed norm $\|\cdot\|$.

\begin{lemma}\label{approx-unit-lemma}
For all $n\in\N$, let $\psi_n : \R^2 \rightarrow [0,\infty)$ be an integrable function supported on $\R^2\left[ 0, \frac{1}{n+1}\right]$ and with $\int_{\R^2} \psi_n = 1$.

If $f : \R^2 \rightarrow [0,\infty)$ is integrable on some ball centered at $0$ in $(\R^2,\|\cdot\|)$, and $f$ continuous at $0$, then:
\begin{equation*}
\lim_{n\rightarrow\infty} \iint_{\R^2} \psi_n(x,y)  f(x,y) \, dx dy = f(0) \text{.}
\end{equation*}
\end{lemma}

\begin{proof}
Let $\delta > 0$ such that $f$ is integrable on $\R^2[0,\delta]$.

Let $\varepsilon > 0$. Since $f$ is continuous at $0$, there exists $\delta_c > 0$ such that $|f(x) - f(0)| \leq \varepsilon$ for all $x \in \R^2[0,\delta_c]$.

Let $N \in \N$ be chosen so that $\frac{1}{N+1} \leq \min \{\delta,\delta_c\}$. For all $n > N$, we first note that since $\psi_n$ is supported on a subset of $\R^2[0,\delta]$, the function $\psi_n f$ is integrable on $\R^2$. Moreover for all $n\geq N$:

\begin{equation*}
\begin{split}
\left |\iint_{\R^2} \psi_n(x,y) f(x,y)\, dx dy - f(0) \right| &\leq \int_{\R^2} |\psi_n(x,y) (f(x,y) - f(0))| \,dx dy \\
&= \iint_{\R^2[0,n^{-1}]} |\psi_n(x,y)| |f(x,y) - f(0)| \, dx dy \\
&\leq \iint_{\R^2[0,n^{-1}]} \psi_n(x,y) \varepsilon \, dx dy \leq \varepsilon \text{.} 
\end{split}
\end{equation*}

Thus we have shown that $\lim_{n\rightarrow\infty} \int_{\R^2} \psi_n(x,y) f(x,y) \, dx dy = f(0)$. 
\end{proof}

We are now ready to prove the second core lemma of this section. We begin with an explanation of the ideas and reasons behind this lemma. 

By a compact operator on a Banach space $(E,\|\cdot\|_{\C^d})$, we mean as usual an operator which maps bounded subsets of $E$ to totally bounded subsets of $E$. 

The map $f \in L^1(\R^2) \mapsto \sigma_{\eth,d}^f$ is a *-representation of the twisted convolution algebra $L^1(\R^2)$ for the convolution product defined for all $f , g \in L^1(\R^2)$ and $x\in \R^2$ by:
\begin{equation*}
f\conv{\eth} g (x) = \int_{\R^2} f(y) g(x-y) \cocycle{\eth}(y,x-y) \,dy
\end{equation*}
and the involution:
\begin{equation*}
f\in L^1(\R^2) \mapsto f^\ast = x \in \R^2 \mapsto \overline{f(-x)}\text{,}
\end{equation*}
as can be directly checked, or is established in \cite{Folland89}. It is an important, well-known fact \cite[Theorem 1.30]{Folland89} that this representation is valued in the algebra of compact operators on $L^2(\R)\otimes\C^d$, and is faithful; the completion of $(L^1(\R^2),\conv{\eth},\ast)$ for the norm $f\in\ell^1(\Z^2) \mapsto \|f\|_{C^\ast(\R^2,\mathrm{e}_\eth)} = \opnorm{\sigma^f_{\eth,1}}{}{L^2(\R)}$ is the entire algebra of compact operators.

The fact that $\sigma^f_{\eth,d}$ is compact as an operator of $L^2(\R)\otimes\C^d$ does not immediately imply that it is compact for the Banach space $\left(\HeisenbergMod{\theta}{p,q,d},\|\cdot\|_{\HeisenbergMod{\theta}{p,q,d}}\right)$ since in general, we only know that $\|\cdot\|_{L^2(R)} \leq\|\cdot\|_{\HeisenbergMod{\theta}{p,q,d}}$. We thus must prove compactness of these operators for our $C^\ast$-Hilbert norm. However, we can extract the essential tools for our work from the expansive work on Laguerre expansion of functions and the study of the Moyal plane. We will prove that, at least when $f$ is a radial function, then we can approximate $\sigma_{\eth,d}^f$ by finite rank operators, in norm. To this end, we need a supply of finite rank operators, which provide a mean to approximate any $\sigma_{\eth,d}^f$ for $f$ radial. The theory of the quantum harmonic oscillator provides us with a well-suited family of finite rank projections, obtained as $\sigma_{\eth,d}^{\psi}$ for $\psi$ a properly scaled Laguerre function \cite[Ch. 1, sec. 9]{Folland89}. 

To obtain the desired approximation result, however, we need to approximate our radial functions in the norm of $L^1(\R^2)$ using functions obtained from Laguerre functions. As Laguerre functions form an orthonormal basis for some $L^2$ space, we certainly do have a Laguerre expansion which converges in some $L^2$ norm, but convergence in $L^1(\R^2)$ is highly not trivial. 

The work of Sundaram Thangaveru in \cite{Thangavelu93} comes to our rescue, however, by proving that we may obtain the desired convergence if we replace the Laguerre expansion series by the sequence of its C{\'e}saro averages. We now formalize our discussion in the next key lemma.

\begin{lemma}\label{compact-operator-lemma}
If $f : \R_+ \rightarrow \R$ is a function such that $r\in\R\mapsto r f(r)$ is Lebesgue integrable, and if we set:
\begin{equation*}
f^\circ : (x,y) \in \R^2 \mapsto f\left(\sqrt{x^2+y^2}\right)\text{,}
\end{equation*}
then the operator $\sigma_{\eth,d}^{f^\circ}$ is a compact operator for the Banach space $\left(\HeisenbergMod{\theta}{p,q,d},\|\cdot\|_{\HeisenbergMod{\theta}{p,q,d}}\right)$.
\end{lemma}

\begin{proof}
Our goal is to write $\sigma_{\eth,d}^{f^\circ}$ as a limit, in the operator norm, of finite rank operators. To this end, let us first assume that $\eth > 0$ and for all $n\in\N$, we let $\psi_\eth^n$ be the $n^{\mathrm{th}}$ Laguerre function defined for all $r \in [0,\infty)$ by:
\begin{equation*}
\psi_\eth^n (r) = \eth \exp\left(- \frac{\pi \eth r^2}{2}\right) \operatorname{L}_n\left( \pi \eth r^2 \right) \text{,}
\end{equation*} 
where $\operatorname{L}_n$ is the $n^{\mathrm{th}}$ Laguerre polynomials, given for all $x\in \R$ by:
\begin{equation*}
\operatorname{L}_n (x) = \sum_{j=1}^n \frac{(-1)^j}{j!} {n \choose n - j}x^j \text{.}
\end{equation*}
Note that these functions are given in \cite[(6.1.17)]{Thangavelu93} for $\eth = \frac{1}{\pi}$. An observation which will be important for us in later proofs is that $\psi_\eth^n = \eth \psi_1^n(\sqrt{\eth}\cdot)$, i.e. we can obtain all the Laguerre functions we are considering via a simple rescaling.

By slight abuse of notation, we denote by $L^p(\R_+,r dr)$ the $p$-Lebesgue space for the measure defined, for all measurable $f : [0,\infty]\rightarrow [0,\infty)$, by $\int_0^\infty f(r) \, r dr$. In particular, note that the inner product of $L^2(\R_+, r dr)$ is given for any two $f, g \in L^2(\R,r dr)$, by:
\begin{equation*}
\inner{f}{g}{L^2(\R_+,r dr)} = \int_0^\infty f(r) \overline{g}(r) \, rdr \text{.}
\end{equation*}

With all these notations set, we define, for each $n\in\N\setminus\{0\}$, the $n^{\mathrm{th}}$ C{\'e}saro sum of the series given by the Laguerre expansion of $f$:
\begin{equation*}
\mathrm{C}_\eth^n(f) = \sum_{j=0}^n \frac{n + 1 -j}{n + 1} \inner{f \, \psi^j_\eth}{\psi^j_\eth}{L^2(\R_+,r dr)}\psi_\eth^j \text{.}
\end{equation*}

Then by the work of S. Thangavelu in \cite[Theorem 6.2.1]{Thangavelu93} --- where our $\psi_\eth^j$ is a rescaled version of the function denoted by $\psi_j^0$ in \cite[Chapter 6]{Thangavelu93} and we use the C{\'e}saro sums for ``$\delta = 1$'' in his notations --- we conclude:

\begin{equation*}
\lim_{n\rightarrow \infty} \left\|\mathrm{C}_\eth^n f  - f \right\|_{L^1(\R_+, r dr)} = 0\text{.}
\end{equation*}

Now, a quick computation shows that for all $n\in\N\setminus\{0\}$:
\begin{equation*}
\left\|(\mathrm{C}_\eth^j(f))^\circ - f^\circ\right\|_{L^1(\R^2)} = \left\|\mathrm{C}_\eth^j(f)  - f \right\|_{L^1(\R_+, r dr)}\text{,}
\end{equation*}
and therefore:
\begin{equation*}
\lim_{n\rightarrow\infty} \left\|(\mathrm{C}_\eth^n(f))^\circ - f^\circ\right\|_{L^1(\R^2)} = 0
\end{equation*}
where of course, $L^1(\R^2)$ stands for the $1$-Lebesgue space with respect to the usual Lebesgue measure on $\R^2$.

By Lemma (\ref{operator-lemma}), writing $\kappa_n = (\mathrm{C}_\eth^n(f))^\circ$ for all $n\in\N$, we then conclude:
\begin{equation*}
\lim_{n\rightarrow\infty} \opnorm{\sigma_{\eth,d}^{\kappa_n} - \sigma_{\eth,d}^{f^\circ}}{}{\HeisenbergMod{\theta}{p,q,d}} = 0\text{.}
\end{equation*}

By construction, $\sigma_{\eth,d}^{\kappa_n}$ is finite rank. Indeed, the operator $\sigma_{\eth,d}^{\kappa_n}$ is a linear combination of the operators $\sigma_{\eth,d}^{(\psi_\eth^j)^\circ}$ with $j\in\{0,\ldots,n\}$. The operators $\sigma_{\eth,d}^{(\psi_\eth^j)^\circ}$ are, in turn, projections on $\C\mathcal{H}_\eth^j \otimes\C^d \subseteq L^2(\R)\otimes\C^d$, where $\mathcal{H}_\eth^n$ is the Hermite function:
\begin{equation*}
\mathcal{H}_\eth^j : t \in \R \mapsto \frac{\left(2\eth\right)^{\frac{1}{4}}}{\sqrt{j! 2^j}} \exp\left(-\frac{ t^2 \sqrt{2 \pi \eth} }{2}\right)\mathrm{H}_j\left(t \sqrt{2\pi\eth}\right)
\end{equation*}
where $\operatorname{H}_j$ is the $j^{\mathrm{th}}$ Hermite polynomial, given for instance by:
\begin{equation*}
\operatorname{H}_j : t \in \R \mapsto (-1)^j \exp(t^2) \frac{\mathrm{d}^j}{\mathrm{d} t^j} \exp(-t^2) \text{.}
\end{equation*}
Indeed, by \cite[p. 65]{Folland89}, the operators $\sigma_{\eth,1}^{(\psi_\eth^j)^\circ}$ are projections on $\C\mathcal{H}_j \subseteq L^2(\R)$ for all $j\in\N$. We note that reassuringly, we will not need the explicit form of the Hermite polynomials or the Laguerre polynomials in our work.

Thus the image of the unit ball $\HeisenbergMod{\theta}{p,q,d}[0,1]$ of $\left(\HeisenbergMod{\theta}{p,q,d},\|\cdot\|_{\HeisenbergMod{\theta}{p,q,d}}\right)$ by $\sigma_{\eth,d}^{\kappa_n}$ is totally bounded in $\left(\HeisenbergMod{\theta}{p,q,d},\|\cdot\|_{\HeisenbergMod{\theta}{p,q,d}}\right)$ for all $n\in\N$, as a bounded subset of a finite dimensional space (as all norms are equivalent in finite dimension, this observation does not depend on $\|\cdot\|_{\HeisenbergMod{\theta}{p,q,d}}$).

Thus $\sigma_{\eth,d}^{f^\circ}$ is compact as the norm limit of compact operators.

We are left to treat the case when $\eth < 0$. We note that for all $(x,y,u) \in \HeisenbergGroup$, we have:
\begin{equation*}
\alpha_{\eth,d}^{x,y,u} = \alpha_{-\eth,d}^{x,-y,-u}\text{.}
\end{equation*}
We thus proceed as above with $-\eth$ in place of $\eth$, and note that $\sigma_{\eth,d}^{\kappa_n} = -\sigma_{-\eth,d}^{\kappa_n}$ since $\kappa_n$ is a radial function. The rest of the proof is left unchanged.
\end{proof}

With Lemma (\ref{compact-operator-lemma}) and Lemma (\ref{dnorm-approx-lemma}), we are now able to prove the desired property for our D-norms:

\begin{lemma}\label{compactness-lemma}
We assume Hypothesis (\ref{hyp-1}). The set:
\begin{equation*}
\modlip{1}{\CDN_\theta^{p,q,d}} = \left\{ \xi \in \HeisenbergMod{\theta}{p,q,d} : \CDN_\theta^{p,q,d}(\xi) \leq 1 \right\}
\end{equation*}
is compact in $\left(\HeisenbergMod{\theta}{p,q,d},\left\|\cdot\right\|_{\HeisenbergMod{\theta}{p,q,d}}\right)$.
\end{lemma}

\begin{proof}
Let $(\psi_n)_{n\in\N}$ be a sequence of smooth functions from $[0,\infty)$ to $[0,\infty)$ such that for all $n\in\N$, the function $\psi_n$ is supported on $\left[-\frac{1}{n+1},\frac{1}{n+1}\right]$ and:
\begin{equation*}
\int_0^\infty \psi_n(r) \, rdr = \frac{1}{2\pi}\text{.}
\end{equation*}

Thus, using the notations of Lemma (\ref{compact-operator-lemma}), we note that:
\begin{equation*}
\int_{\R^2}\psi_n^\circ = \int_{-\frac{\pi}{2}}^{\frac{\pi}{2}}\int_0
^\infty \psi_n(r) \,r dr d\theta = \frac{2\pi}{2\pi} = 1\text{.}
\end{equation*}

Let $\varepsilon > 0$ be given. By Lemma (\ref{approx-unit-lemma}), we have:
\begin{equation*}
\lim_{n\rightarrow\infty} \iint_{\R^2} \psi_n^\circ(x,y) \|(x,y)\| \, dx dy = 0 \text{.}
\end{equation*}
Thus, there exists $N\in \N$ such that for all $n\geq N$, the following inequality holds:
\begin{equation*}
\iint_{\R^2} \psi_n^\circ(x,y) \|(x,y)\| \, dx dy < \frac{\varepsilon}{4\pi\eth}
\end{equation*}

We may thus apply Lemma (\ref{dnorm-approx-lemma}) to conclude that for all $\xi \in \modlip{1}{\CDN_\theta^{p,q,d}}$ and $n \geq N$:
\begin{equation*}
\left\| \xi - \sigma_{\eth,d}^{\psi_n^\circ}\xi\right\|_{\HeisenbergMod{\theta}{p,q,d}} \leq \frac{\varepsilon}{2} \text{.}
\end{equation*}

Now, $\sigma_{\eth,d}^{\psi_n^\circ}$ is compact in $\left(\HeisenbergMod{\theta}{p,q,d},\|\cdot\|_{\HeisenbergMod{\theta}{p,q,d}}\right)$ by Lemma (\ref{compact-operator-lemma}), and $\modlip{1}{\CDN_\theta^{p,q,d}}$ is bounded for $\|\cdot\|_{\HeisenbergMod{\theta}{p,q,d}}$ by construction. Thus the image of $\modlip{1}{\CDN_\theta^{p,q,d}}$ by $\sigma_{\eth,d}^{\psi_n^\circ}$ is totally bounded in $\left(\HeisenbergMod{\theta}{p,q,d},\|\cdot\|_{\HeisenbergMod{\theta}{p,q,d}}\right)$ for all $n\in\N$. In particular, there exists a $\frac{\varepsilon}{2}$-dense subset $\module{B}_\varepsilon$ in $\sigma_{\eth,d}^{\psi_N^\circ}\modlip{1}{\CDN_\theta^{p,q,d}}$. 

Consequently, if $\xi\in\modlip{1}{\CDN_\theta^{p,q,d}}$, then there exists $\eta \in \module{B}_\varepsilon$ such that:
\begin{equation*}
\left\|\eta - \sigma_{\eth,d}^{\psi_N^\circ}\xi\right\|_{\HeisenbergMod{\theta}{p,q,d}} \leq \frac{\varepsilon}{2}\text{.}
\end{equation*}
Thus $\left\|\xi - \eta\right\|_{\HeisenbergMod{\theta}{p,q,d}} \leq \varepsilon$.

We thus conclude that $\modlip{1}{\CDN_\theta^{p,q,d}}$ is totally bounded.

Moreover, for all $(x,y) \in \R^2$, the map $\xi \mapsto \frac{\|\alpha_{\eth,d}^{x,y,\frac{x y}{2}}\xi - \xi\|_{\HeisenbergMod{\theta}{p,q,d}}}{2\pi|\eth|\|(x,y)\|}$ is continuous, and thus $\CDN_\theta^{p,q,d}$ is lower semi-continuous with respect to $\|\cdot\|_{\HeisenbergMod{\theta}{p,q,d}}$. Hence $\modlip{1}{\CDN_\theta^{p,q,d}} = \left(\CDN_\theta^{p,q,d}\right)^{-1}((-\infty,1])$ is closed. Since $\HeisenbergMod{\theta}{p,q,d}$ is complete and $\modlip{1}{\CDN_\theta^{p,q,d}}$ is closed and totally bounded, it is in fact compact, as desired.
\end{proof}

We summarize the results of this section with the following theorem announcing that indeed, we have defined D-norms on Heisenberg modules, turning them into {\gQVB s} over quantum $2$-tori.

\begin{theorem}\label{HeisenbergMod-dnorm-thm}
Let $\HeisenbergMod{\theta}{p,q,d}$ be the Heisenberg module over $\qt{\theta}$ for some $\theta\in\R$,  $p\in\Z$, $q \in \N\setminus\{0\}$ and $d\in q\N\setminus\{0\}$. Let $\eth = \theta - \frac{p}{q}$ and assume $\eth \not= 0$. Let $\|\cdot\|$ be a norm on $\R^2$. If we set, for all $\xi \in \HeisenbergMod{\theta}{p,q,d}$:
\begin{equation*}
\CDN_\theta^{p,d,q}(\xi) = \sup\left\{ \norm{\xi}{\HeisenbergMod{\theta}{p,q,d}}, \frac{\left\|\sigma_{\eth,d}^{x,y}\xi - \xi\right\|_{\HeisenbergMod{\theta}{p,q,d}}}{2\pi|\eth|\|(x,y)\|}  : (x,y) \in \R^2\setminus\{0\} \right\} \text{,}
\end{equation*}
and for all $a\in\qt{\theta}$:
\begin{equation*}
\Lip_\theta(a) = \sup\left\{\frac{\left\|\beta_\theta^{\exp(ix),\exp(iy)}a - a\right\|_{\qt{\theta}}}{\|(x,y)\|} : (x,y)\in\R^2\setminus\{0\} \right\}
\end{equation*}
then $\left(\HeisenbergMod{\theta}{p,q,d}, \inner{\cdot}{\cdot}{\HeisenbergMod{\theta}{p,q,d}}, \CDN_\theta^{p,d,q}, \qt{\theta}, \Lip_\theta\right)$ is a Leibniz {\gQVB}.
\end{theorem}

\begin{proof}

Proposition (\ref{DLeibniz-prop}) proves that $\CDN_\theta^{p,q,d}$ is a norm on a dense subspace of $\HeisenbergMod{\theta}{p,q,d}$ which satisfies the inner and modular quasi-Leibniz inequalities and, by construction, $\CDN_\theta^{p,q,d} \geq \|\cdot\|_{\HeisenbergMod{\theta}{p,q,d}}$.

Lemma (\ref{compactness-lemma}) moreover gives us that $\modlip{1}{\CDN_\theta^{p,q,d}}$ is compact for $\|\cdot\|_{\HeisenbergMod{\theta}{p,q,d}}$.
\end{proof}

\bibliographystyle{amsplain}
\bibliography{../thesis}
\vfill

\end{document}